\documentclass[12pt]{amsart}
\usepackage[a4paper,top=1.5in,bottom=1in,left=1in,right=1in,marginparwidth=1.75cm]{geometry}
\usepackage{amsmath,amssymb,mathtools}
\usepackage{enumitem,makecell,caption,longtable,wrapfig}
\usepackage[colorlinks=true, allcolors=blue]{hyperref}
\usepackage{tikz}
\usetikzlibrary{tqft}

\newcommand{\mc}[1]{\mathcal{#1}}
\newcommand{\mbf}[1]{\mathbf{#1}}
\newcommand{\mr}[1]{\mathrm{#1}}

\newcommand{\gen}[1]{\langle {#1} \rangle}

\newcommand{\C}{\mathbb{C}}
\newcommand{\Zen}[1]{\mathbf{Z}(#1)}
\newcommand{\Cen}[2]{\mathbf{C}_{#1}(#2)}
\newcommand{\Norm}[2]{\mathbf{N}_{#1}(#2)}
\newcommand{\nrm}{\trianglelefteq}
\newcommand{\Irr}[1]{\mathrm{Irr}(#1)}
\newcommand{\IBr}[1]{\mathrm{IBr}(#1)}

\newcommand{\til}[1]{\widetilde{#1}}
\newcommand{\q}[1]{\overline{#1}}

\theoremstyle{definition}

\newtheorem{theorem}{Theorem}[section]

\newtheorem{lemma}[theorem]{Lemma}

\newtheorem{remark}[theorem]{Remark}

\usepackage{amsfonts, amsmath, amssymb, mathrsfs, bm, latexsym, stmaryrd, array, mathtools, bold-extra}
\usepackage{etaremune}
\usepackage{enumitem}
\usepackage[pagewise]{lineno}  \nolinenumbers 

\title[Invariants of finite groups arising in TQFT]{On the invariants of finite groups arising in a topological quantum field theory}

\author{Christopher A. Schroeder and Hung P. Tong-Viet}
\date{\today}

\begin{document}

\begin{abstract}


In this paper, we investigate structural properties of finite groups that are detected by certain group invariants arising from Dijkgraaf–Witten theory, a topological quantum field theory, in one space and one time dimension. In this setting, each finite group $G$ determines a family of numerical invariants associated with closed orientable surfaces, expressed in terms of the degrees of the complex irreducible characters of $G$. These invariants can be viewed as natural extensions of the commuting probability 
$d(G)$, which measures the probability that two randomly chosen elements of $G$ commute and has been extensively studied in the literature. By analyzing these higher-genus analogues, we establish new quantitative criteria relating the values of these invariants to key structural features of finite groups, such as commutativity, nilpotency, supersolvability and solvability. Our results generalize several classical theorems concerning the commuting probability, thereby linking ideas from finite group theory and topological quantum field theory.
\end{abstract}

\maketitle

\section{Introduction}
There is a long history of using invariants to deduce structural properties of finite groups. These invariants are often constructed by counting objects associated with the group, such as elements, conjugacy classes or irreducible characters. In this paper, we take a different approach. Given the remarkable compatibility between mathematics and physics, it is reasonable to expect that group invariants arising naturally in physics are useful for characterizing group structure. Here, we consider Dijkgraaf--Witten theory, a topological quantum field theory (TQFT), in one space dimension and one time dimension. In short, such a TQFT determines a commutative Frobenius algebra 
that gives rise to invariants of surfaces. If we assume that this algebra is the center $\mbf{Z}(\C G)$ of the complex group algebra of a finite group $G$, then the 
invariant of a \mbox{genus-$h$} surface takes the form
\begin{align*}
    \mc{Q}_h(G) = \sum_{\chi \in \mr{Irr}(G)} \Big(\frac{|G|}{\chi(1)}\Big)^{2h-2},
\end{align*}
where $\Irr{G}$ is the collection of complex irreducible characters of $G$. Note that $\mc{Q}_0(G) = 1/|G|$ and $\mc{Q}_1(G) = k(G)$, where $k(G)$ is the number of conjugacy classes of $G$ (equivalently, the number $|\Irr{G}|$ of irreducible complex characters). 
If we consider the invariants $\mc{Q}_h(G)$ as arising from a sequence of increasingly complicated experiments, we might imagine that we first measure $\mc{Q}_0(G)$, then $\mc{Q}_1(G)$, and so on. With this in mind, it is natural to consider the following scaled invariants, which, as we will show, are very well-behaved:
\begin{align*}
    q_h(G) := 
    \frac{1}{|G|} \sum_{\chi \in \Irr{G}} \Big(\frac{1}{\chi(1)}\Big)^{2h-2}.
\end{align*}
Now note that $q_0(G)=1$ and $q_1(G)=k(G)/|G|$, where the latter is the so-called commuting probability $d(G)$, which was first considered by Joseph \cite{J69,J77}, then by Gustafson~\cite{G73} and many others. Viewed in this way, the ``quantum invariants'' $q_h(G)$ are generalizations of the commuting probability, for which the following structure criteria are known:
\begin{enumerate}[nolistsep,label=(\alph*)]
    \item If $d(G) > d(\mr{D}_8) = 5/8$, then $G$ is abelian (Gustafson \cite{G73}).
    \item If $d(G)>d(\mr{S}_3)=1/2$, then $G$ is nilpotent (Lescot \cite{Le87}).
    \item If $d(G) > d(\textrm{A}_4) = 1/3$, then $G$ is supersolvable (Barry, MacHale, N\'{\i}~Sh\'{e} \cite{BMN}).
    \item If $d(G) > d(\mr{A}_5)=1/12$, then $G$ is solvable (Dixon \cite{D73}).
\end{enumerate}

In this paper, we investigate which structural properties of finite groups are witnessed by the invariants $q_h(G)$. Our main results establish broad generalizations of the classical structure criteria above to arbitrary genus $h\ge 1$. The bounds in our first theorem are clearly best possible.


\begin{theorem}\label{t:qh-criteria}
Let $G$ be a finite group, and let $h$ be a positive integer.
\begin{enumerate}[nolistsep,label=(\alph*)]
    \item If $q_h(G) > q_h(\mr{D}_8)$, then $G$ is abelian.
    \item If $q_h(G)>q_h(\mr{S}_3)$, then $G$ is nilpotent.
    \item If $q_h(G) > q_h(\textrm{A}_4)$, then $G$ is supersolvable.
    \item If $q_h(G) > q_h(\mr{A}_5)$, then $G$ is solvable.
\end{enumerate}
\end{theorem} 


Guralnick and Robinson showed in \cite[Lemma~2(x)]{G06} that if $p$ is a prime and \mbox{$d(G)>1/p$,} then $G$ has a normal Sylow $p$-subgroup. Our next result generalizes this $p$-closed criterion to all genera.

\begin{theorem}\label{t:pclosed}
Let $G$ be a finite group, let $p$ be a prime, and let $h$ be a positive integer. If $q_h(G) > \gamma(h,p)$, where $\gamma(h,p)=(1+1/p^{2h-1})/(p+1)$, then~$G$ has a normal Sylow $p$-subgroup.
\end{theorem}

This criterion is the best possible. For example, if $p=2^f-1$ is a Mersenne prime, then the Frobenius group $G=(C_2)^f \rtimes C_p$ satisfies $q_h(G)=\gamma(h,p)$ and does not have a normal Sylow $p$-subgroup. Although Theorem~\ref{t:pclosed} implies Theorem~\ref{t:qh-criteria}(b) and much of part (c), we will in fact use the latter to prove this $p$-closed criterion.

Given a prime $p$, it is also natural to consider a $p$-local version of the invariants $q_h(G)$ by summing over irreducible $p$-Brauer characters instead of irreducible complex characters. We define
\begin{align*}
    q_{h,p'}(G) = \frac{1}{|G|_{p'}} \sum_{\varphi \in \IBr{G}} \frac{1}{\varphi(1)^{2h-2}}.
\end{align*}
Note that $q_{1,p'}(G)=k_{p'}(G)/|G|_{p'}$, where $k_{p'}(G)$ is the number of irreducible $p$-Brauer characters (equivalently, the number of $p$-regular conjugacy classes). This invariant is denoted by $d_{p'}(G)$ in the literature, and it was shown in \cite[Theorem~1.1]{S24} that if $p$ is an odd prime and \mbox{$d_{p'}(G) > 1/(p-1)$}, then $G$ is $p$-solvable. This result is best possible for $p>3$ since \mbox{$d_{p'}(\mr{PSL}_2(p))=1/(p-1)$}. 
Our next theorem extends this result to all values of the genus.

\begin{theorem}\label{t:plocal}
Let $G$ be a finite group, let $h$ be a positive integer, and let $p$ be an odd prime. If $q_{h,p'}(G)>\alpha(h,p)/(p-1)$, where $\alpha(h,p)= (2^{2h-2}+\sqrt{p-1})/(2^{2h-2}\sqrt{p-1})$, then $G$ is $p$-solvable.
\end{theorem}

Note that $\alpha(h,p)<1$ when both $h > 1$ and $p>2$. Theorem \ref{t:plocal} is not best possible for $h=1$, in which case, again, the best bound is $1/(p-1)$ from~\cite{S24}. We suspect, but have not proved, that the result is also not best possible for $h>1$. 
Our proof relies on a lower bound for $k_{p'}(G)$ when $G$ is not $p$-solvable. At the time of writing, the best known bound is $k_{p'}(G) > \sqrt{p-1}$, which was proved by Hung and Mar{\'o}ti in~\cite{NM22}. An improvement in this bound would lead directly to an improvement in Theorem \ref{t:plocal}.

Finally, it is natural to consider an invariant $\widetilde{q}_h(G)$ dual to $q_h(G)$ by summing over the sizes of the conjugacy classes $C \in \mr{Cl}(G)$ of a finite group $G$, where $\mr{Cl}(G)$ denotes the set of all conjugacy classes of $G$. We define the invariant 
\begin{align*}
    \widetilde{q}_h(G) = \frac{1}{|G|} \sum_{C \in \mr{Cl}(G)} \frac{1}{|C|^{h-1}}.
\end{align*}
Note that $\til{q}_0(G)=1$ and $\til{q}_1(G)=d(G)$. From the definition, $\til{q}_{h}(G) \ge \til{q}_{h+1}(G)$ for all  integers $h\ge 0$. We again prove generalizations of the structure criteria witnessed by the commuting probability to all values of the genus $h$.

\begin{theorem}\label{t:ch-criteria}
Let $G$ be a finite group, and let $h$ be a positive integer.
\begin{enumerate}[nolistsep,label=(\alph*)]
    \item If $\til{q}_h(G) > \til{q}_h(\mr{D}_8)$, then $G$ is abelian.
    \item If $\til{q}_h(G)>\til{q}_h(\mr{S}_3)$, then $G$ is nilpotent.
    \item If $\til{q}_h(G) > \til{q}_h(\textrm{A}_4)$, then $G$ is supersolvable.
    \item If $\til{q}_h(G) > \til{q}_h(\mr{A}_5)$, then $G$ is solvable.
\end{enumerate}
\end{theorem} 

We also prove a dual criterion for the existence of a normal Sylow $p$-subgroup. 

\begin{theorem}\label{t:pclosed-dual}
Let $G$ be a finite group, let $p$ be a prime and let $h$ be a positive integer. If $\til{q}_h(G)>\til{\gamma}(h,p)$, where $\til{\gamma}(h,p)= (1+1/p^{h-1}+(p-1)/(p+1)^{h-1})/(p(p+1))$, then $G$ has a normal Sylow $p$-subgroup.
\end{theorem}

This criterion is best possible, as again the Frobenius group $G=(C_2)^f \rtimes C_p$ with $p=2^f-1$ a Mersenne prime satisfies $\til{q}_h(G) = \til{\gamma}(h,p)$.

We conclude this Introduction with some remarks.

\begin{remark}
The content of Theorems \ref{t:qh-criteria}(a)--(d) and \ref{t:ch-criteria}(a)--(d) is that the group-theoretic properties of commutativity, nilpotency, supersolvability and solvability are witnessed by the value of $q_h(G)$ or $\til{q}_h(G)$ for a particular ``minimal'' group $G$ for all $h \ge 1$. For example, Theorem \ref{t:qh-criteria}(d) says that solvability is witnessed by $\mr{A}_5$ for all values of the genus $h$. Of course, our theorems would have trivial content if the invariants preserved the ordering of finite groups for all $h$; that is, if for all finite groups $G$ and $H$ it were true that $q_h(G) > q_h(H)$ for some $h>1$ implied that $d(G) > d(H)$, then all our theorems would follow trivially from the corresponding result for~$d(G)$. But this is not the case. For example, if $G=\mr{S}_6$ and $H=\mr{PGL}_2(9)$, then $q_h(G)>q_h(H)$ and $\til{q}_h(G)>\til{q}_h(H)$ for all $h>1$, but $d(G)=d(H)$. 

Our next example illustrates a general phenomenon: Notice that for a finite group~$G$, the invariant $q_h(G)$ approaches $1/|G'|$ as $h$ approaches infinity, where $G'$ denotes the derived subgroup of $G$. Therefore, if $d(G)>d(H)$ for some finite groups $G$ and $H$, but $1/|G'|<1/|H'|$, then there must exist some positive integer $N$ such that $q_h(G)<q_h(H)$ for all $h>N$. For example, let $G=\mr{A}_5$, and let $H$ be an extraspecial $p$-group of order $p^3$ for the prime $p=59$. Then $d(G)=1/12$ and $1/|G'|=1/60$, while $d(H)=(p^2+p-1)/p^3$ and $1/|H'|=1/p$. Then $d(G)>d(H)$, while $1/|G'|<1/|H'|$, and the preceding observation applies. In fact, a calculation shows that $q_h(G)>q_h(H)$ for $h=1,2,3$, while $q_h(G)<q_h(H)$ for $h \ge 4$. 

For the dual invariant, take $G=\mr{A}_5$ and let $H$ be an extraspecial $p$-group of order $11^3$. Then $d(G)<d(H)$ yet $\til{q}_h(G)>\til{q}_h(H)$ for all integers $h\ge 2.$
\end{remark}

\begin{remark}
By examining our proofs, it is easy to see that our results hold not only when the exponent in $q_h(G)$ is the negative of the Euler characteristic $2-2h$ of a genus-$h$ surface, but for any real number $s \ge 1$. In this way, our results are related to the so-called ``zeta function'' $\zeta_G(s) = \sum_{\chi \in \Irr{G}} \chi(1)^{-s}$ of a finite group $G$, which was used by Liebeck and Shalev to prove various results on Fuchsian groups, coverings of Riemann surfaces, subgroup growth, random walks and representation varieties \cite{LiSh04, LiSh05, LiSh05ii}. We also note that the zeta function for conjugacy class sizes was studied in~\cite{LiSh05} for finite groups of Lie type, with applications to base sizes of primitive permutation groups. Since our motivation stems from topological quantum field theories, we have stated and proved our results in this special case.
\end{remark}

\begin{remark}
It is known that $d(G) \le d(N) d(G/N)$ when $N \nrm G$ \cite[Lemma~2(ii)]{G06}. Our proofs would be simplified if $q_h(G) \le q_h(N) q_h(G/N)$ for all $h$. We have not been able to prove this statement, however. If it were true, then it must of course be true for $N=\Zen{G}$. But this already seems to be a hard question about projective character degrees: Letting $Z=\Zen{G}$,
\begin{align*}
    q_h(G) = \frac{1}{|G/Z||Z|} \sum_{\lambda \in \Irr{Z}} \sum_{\chi \in \Irr{G|\lambda}} \frac{1}{\chi(1)^{2h-2}}.
\end{align*}
So to prove that $q_h(G) \le q_h(G/Z)$, it would suffice to show, for example, that
\begin{align*}
    \sum_{\chi \in \Irr{G|\lambda}} \frac{1}{\chi(1)^{2h-2}} \le \sum_{\chi \in \Irr{G/Z}} \frac{1}{\chi(1)^{2h-2}};
\end{align*}
that is, we want to show the sum over projective character degrees of $G/Z$ is less than the sum over ordinary character degrees. We know that they are fewer in number: $|\Irr{G|\lambda}| \le k(G/Z)$ as a consequence of Gallagher's result on ``$\theta$-good'' conjugacy classes (see Theorem~1.19 and Corollary~1.23 in \cite{I18}). 
\end{remark}


\begin{remark}
The ratios $|G|/\chi(1)$ for $\chi \in \mr{Irr}(G)$ are the character degree quotients originally considered by Chillag and Herzog in \cite{CH}. In order to be compatible with group quotients, many authors have considered instead the character codegrees $[G:\mr{ker}(\chi)]/\chi(1)$. Note that if $G$ is simple, then the character codegrees and character degree quotients of nontrivial characters coincide.
\end{remark}

\begin{remark}
Our ``dual invariant'' $\til{q}_h(G)$ also has a physical interpretation in the context of lattice Hamiltonians for Yang--Mills gauge theories with finite gauge group~\cite{Mar23}. These theories are motivated by the potential for quantum simulation of gauge theories.
\end{remark}


\medskip
Our paper is structured as follows. As our methods are purely group-theoretic, we defer an introduction to Dijkgraaf--Witten theory to Section \ref{s:tqft}. 
In Section \ref{s:prelim}, we make some remarks and collect some known facts that will be needed in the rest of the paper. In Section~\ref{s:criteria}, we prove Theorems \ref{t:qh-criteria}(a)--(d) and Theorem~\ref{t:plocal}. In Section~\ref{sec:p-closed}, we prove Theorem~\ref{t:pclosed}. In Section~\ref{s:ch-criteria} we prove Theorems~\ref{t:ch-criteria}(a)--(d) and Theorem~\ref{t:pclosed-dual}. Only the proof of Theorem \ref{t:plocal} requires us to cite results that depend on the classification of finite simple groups. We do use Brauer's $k(GV)$-theorem in the proof of Theorem~\ref{t:qh-criteria}(c), but there the group $G$ is solvable, and the proof in that case does not rely on the classification. All groups considered in this paper are assumed to be finite.

\medskip\noindent\textbf{Notation}. 
Our notation is standard. We follow \cite{I08} for finite group theory and \cite{I94} for the character theory of finite groups.
Let $n$ be a positive integer. We denote by $\mr{A}_n$ and $\mr{S}_n$ the alternating and symmetric groups of degree $n$, respectively. The dihedral group of order $2n$ is denoted by $\mr{D}_{2n}$, and a cyclic group of order $n$ by $C_n$. The special linear, projective general linear, and projective special linear groups of degree $n$ over the finite field of size $q$ are denoted by $\mr{SL}_n(q)$, $\mr{PGL}_n(q)$, and $\mr{PSL}_n(q)$, respectively.
For a finite group $G$ and a subset $X \subseteq G$, we write $\Cen{G}X$ and $\Norm{G}X$ for the centralizer and normalizer of $X$ in $G$, respectively. The derived subgroup of $G$ is denoted by $G'$, and $\mbf{Z}(G)$ denotes the center of $G$. The set of complex irreducible characters of $G$ is denoted by $\mr{Irr}(G)$, and we let $k(G)$ be the number of conjugacy classes of $G$. We will repeatedly use that $k(G)=|\mr{Irr}(G)|$. We denote by $\mr{Cl}(G)$ the set of all conjugacy classes of $G$. If $N \le G$ and $\lambda \in \mr{Irr}(N)$, then $\mr{Irr}(G \vert \lambda)$ denotes the set of all irreducible characters of $G$ lying above $\lambda$. If $p$ is a prime, the set of irreducible $p$-Brauer characters of $G$ is denoted by $\mr{IBr}(G)$. All other notation is standard.

\section{Preliminaries.}\label{s:prelim}

As noted in the introduction, the invariants $q_h(G)$ can be viewed as generalizations of the commuting probability $d(G)=k(G)/|G|$ first considered by Joseph \cite{J69,J77}, then by Gustafson~\cite{G73}. Our proofs are based on the extensive literature on this invariant. Our first lemma collects some of the facts that are already known.

\begin{lemma}\label{l:dG} 
Let $G$ be a finite group.
\begin{enumerate}[label=(\alph*)]
    \item We have $q_1(G)=d(G)$ and $q_h(G) \ge q_{h+1}(G)$ for all $h \ge 0$.
    \item If $G$ is perfect and $d(G)>1/20$, then $G \cong \mr{A}_5$ or $G \cong \mr{SL}_2(5)$.
    \item If $p$ is a prime such that a Sylow $p$-subgroup of $G$ is not normal, then $d(G) \le 1/p$.
    \item If $G$ is a nonabelian $p$-group for some prime $p$, then $d(G) \le 1/p+1/p^2$.
    \item If $G$ is nilpotent with a nonabelian Sylow $p$-subgroup for some prime $p$, then $d(G)<1/p$ or $|G'|=p$.
    \item Whenever $N \nrm G$, we have $d(G) \le d(G/N)d(N)$. In particular, we always have $d(G)\le d(G/N)$.
\end{enumerate}
\end{lemma}

\begin{proof}
Part (a) follows immediately from the definition using $|\Irr{G}|=k(G)$ and \mbox{$\chi(1) \ge 1$} for all $\chi \in \Irr{G}$. Part (b) is \cite[Proposition 1.6]{C25}. Part (c) is \cite[Lemma 2(x)]{G06}. Part~(d) is \cite[Lemma 2(viii)]{G06}. Part (e) is \cite[Lemma 2(ix)]{G06} and part (f) is \cite[Lemma~2(ii)]{G06}.
\end{proof}

\noindent Our next lemma shows that $q_h(G)$ is monotone with respect to subgroups $H \le G$ for all $h \ge 1$. 

\begin{lemma}\label{l:mono}
Let $G$ be a finite group, and let $h$ be any positive integer. If $H \le G$, then \mbox{$q_h(H) \ge q_h(G)$}.
\end{lemma}

\begin{proof}
We have
\begin{align*}
    q_h(G) &= \frac{1}{|G|} \sum_{\chi \in \Irr{G}} \frac{1}{\chi(1)^{2h-2}} 
    \le \frac{1}{|G|} \sum_{\psi \in \Irr{H}} \sum_{\chi \in \Irr{G|\psi}} \frac{1}{\chi(1)^{2h-2}} \\
    & \le \frac{1}{|G|} \sum_{\psi \in \Irr{H}} \frac{[G:H]}{\psi(1)^{2h-2}} 
    =\frac{1}{|H|} \sum_{\psi \in \Irr{H}} \frac{1}{\psi(1)^{2h-2}} = q_h(H).
\end{align*}
The first inequality follows because we have enumerated all the $\chi \in \Irr{G}$ (possibly multiple times). The characters in $\Irr{G|\psi}$ are the constituents of $\psi^G$ 
by Frobenius reciprocity, and $\psi^G(1)=[G:H]\psi(1)$. The sum $\sum_{\chi \in \Irr{G|\psi}} 1/\chi(1)^{2h-2}$ is as large as possible when all the degrees are as small as possible and $|\Irr{G|\psi}|$ is as large as possible. Since $\psi$ lies under $\chi$, $\chi(1) \ge \psi(1)$. If equality holds for all $\chi \in \Irr{G|\psi}$, then $\chi(1)$ is as small as possible and $|\Irr{G|\psi}|=[G:H]$ is as large as possible, which yields the second inequality. 
\end{proof}

We will need the following result, which slightly improves Lemma 2(x) in \cite{G06} for nonabelian simple groups.

\begin{lemma}\label{lem1}
Let $G$ be a finite nonabelian simple group, and let $p$ be a prime dividing~$|G|$.  Then $d(G)\leq {1}/{(p+1)}$.
\end{lemma}

\begin{proof} Suppose by contradiction that $d(G)>1/(p+1)$. Since $G$ is a finite nonabelian simple group, by \cite{D73} we know that $d(G)\leq 1/12.$ Hence $1/(p+1)<1/12$, which implies that $p\ge 13$ as $p$ is a prime. Let $P$ be a Sylow $p$-subgroup of $G$. Suppose that \mbox{$[G:\Norm{G}{P}] \leq p$}. Since $[G:\Norm{G}{P}]$ is coprime to $p$, we deduce that $[G:\Norm{G}{P}]\leq p-1$. Since $G$ is nonabelian simple, the core of $\Norm{G}{P}$ in $G$ is trivial, so $G$ embeds into $\textrm{S}_{p-1}$, the symmetric group of degree $p-1$. But then $|G|$ is not divisible by $p$, which is a contradiction. Thus $[G:\Norm{G}{P}]\ge p+1$. By Burnside's normal $p$-complement theorem (\cite[Theorem~IV.2.6]{Hu67}), we have $\Norm{G}{P}\neq \Cen{G}{P}$ and hence $|\Norm{G}{P}|\ge 2p$. Combining these two inequalities, we obtain
\begin{align*}
    |G|=[G:\Norm{G}{P}]\cdot |\Norm{G}{P}|\ge 2p(p+1).
\end{align*}
Let $d$ be the minimal degree of a nontrivial complex irreducible character of $G$. By \cite[Theorem~1]{FT}, we have $d\ge (p-1)/2$.
By \cite[Lemma 2(vi)]{G06}, noting that $G'=G$, we have \[d(G)<\frac{1}{d^2}+\Big(1-\frac{1}{d^2}\Big)\frac{1}{|G'|}\leq \frac{4}{(p-1)^2}+\frac{1}{2p(p+1)}.\] Since $d(G)>1/(p+1)$, it follows that \[\frac{1}{p+1}<\frac{4}{(p-1)^2}+\frac{1}{2p(p+1)}.\] Simplifying this inequality gives $2p^3-13p^2-4p<1$, or equivalently $2p^3-13p^2-4p\leq 0$. This can be rewritten as $(2p-13)p\leq 4$ which is impossible for $p\ge 13.$
Hence, the assumption $d(G)>1/(p+1)$ leads to a contradiction, and the lemma follows.
\end{proof}

We also need the following easy lemma due to Lescot \cite{Le87}, whose proof we include for completeness.

\begin{lemma}\label{l:G'}
If $d(G)>1/4$, then $|G'| \le 3/(4d(G)-1)$.
\end{lemma}
\begin{proof}
Since $|G| = \sum_{\chi \in \Irr{G}} \chi(1)^2$ and the number of linear characters is $[G:G']$, we have
\begin{align*}
    |G| -[G:G'] = \sum_{\chi(1) > 1} \chi(1)^2 \ge 4(k(G)-[G:G']).
\end{align*}
Dividing both sides by $|G|$ and rearranging yields the claim.
\end{proof}

\begin{remark}\label{r:bound}
We will repeatedly use the fact that if the smallest degree of a nonlinear irreducible character of $G$ is $n$, then $d(G)$ and $q_h(G)$ can be bounded in terms of $n$ and~$|G'|$, with equality if all nonlinear characters have degree $n$:
\begin{align*}
    q_h(G) \le \frac{1}{|G|}\Big( [G:G']+\frac{k(G)-[G:G']}{n^{2h-2}} \Big) = \frac{1}{|G'|}\Big(1-\frac{1}{n^{2h-2}}\Big)+\frac{d(G)}{n^{2h-2}}.
\end{align*}
We will derive similar inequalites for $\til{q}_h(G)$ in the course of our proofs.
\end{remark}

We now turn to the properties of a finite group $G$ that are witnessed by the dual invariant $\til{q}_h(G)=1/|G| \sum_{C \in \mr{Cl}(G)} 1/|C|^{h-1}$, where $\mr{Cl}(G)$ is the collection of conjugacy classes of $G$. Whereas we showed that $q_h(G)$ is monotone with respect to subgroups in Lemma~\ref{l:mono}, we now prove, in a nicely dual result, that $\til{q}_h(G)$ is monotone with respect to quotients.

\begin{lemma}\label{lem:quotient}
Let $G$ be a finite group, let $N \nrm G$, and let $h$ be any positive integer. Then $\til{q}_h(G) \le \til{q}_h(G/N)$.
\end{lemma}

\begin{proof}
Let $\q{G}=G/N$ and use the `bar' notation. Let $\pi:G \rightarrow G/N$ denote the canonical homomorphism. Given a conjugacy class $D \in \mr{Cl}(\q{G})$, the preimage $\pi^{-1}(D) \subseteq G$ is closed under conjugation, so it is a union of conjugacy classes $C_i$ of $G$, say
\begin{align*}
    \pi^{-1}(D) = C_1 \cup C_2 \cup \cdots \cup C_{k}.
\end{align*}
Let $S(D)=\{C_1,\dots,C_k\}$. Choosing some $x_i \in C_i$ for $i=1,\dots,k$, we have $|C_i| = [G:\Cen{G}{x_i}]$. Similarly, $|D|=[\q{G}:\Cen{\q{G}}{\q{x}_i}]$ for all $i=1,\dots,k$ since $\pi(x_i) \in D$. Writing $\Cen{\q{G}}{\q{x}_i}=H/N$ for some subgroup $H \le G$, we have $|D|=[G:H]$. Note that $\Cen{G}{x_i} \le H$. Therefore,
\begin{align*}
    |C_i|=[G:\Cen{G}{x_i}]=[G:H][H:\Cen{G}{x_i}] \ge |D|.
\end{align*}
Furthermore, $|\pi^{-1}(D)|=|D||N|$ since $\pi$ is an $|N|$-to-$1$ map. Since $|C_i| \ge |D|$, we have
\begin{align*}
    |N||D| = \sum_{i=1}^{k} |C_i| \ge k |D|,
\end{align*}
and so $k \le |N|$.

Each conjugacy class of $G$ lies in the preimage of exactly one conjugacy class of $\q{G}$. Now we calculate
\begin{align*}
    \til{q}_h(G) = \frac{1}{|G|} \sum_{D \in \mr{Cl}(\q{G})} \; \sum_{C \in S(D)} \frac{1}{|C|^{h-1}} \le \frac{1}{|G|} \sum_{D \in \mr{Cl}(\q{G})} \frac{|N|}{|D|^{h-1}} = \til{q}_h(\q{G}),
\end{align*}
which is what we wanted to show.
\end{proof}

The following easy lemma is our main tool for proving the structure criteria in Theorem~\ref{t:ch-criteria}.

\begin{lemma}\label{l:c-inequality}
Let $G$ be a finite group, and let $n$ be the  smallest size of the noncentral conjugacy classes of $G$. Let $h\ge 2$ be an integer, and assume that $G$ is not abelian. Let $N$ be an integer such that $d(G)/\til{q}_h(G) \le N$. Then one of the following occurs:
\begin{enumerate}
    \item $n^{h-1} \le N$; or
    \item $[G:\Zen{G}] \leq  \dfrac{N}{\til{q}_h(G) (N+1)-d(G)} \le \dfrac{d(G)}{\til{q}_h(G)^2}$.
\end{enumerate}
\end{lemma}

\begin{proof} Since $G$ is not abelian, $\til{q}_h(G) < d(G)$.
If $n^{h-1}\leq N$, then part (1) is satisfied. Assume that $n^{h-1} > N$.  So $n^{h-1}>N\ge d(G)/\til{q}_h(G)$. It follows that \mbox{$\til{q}_h(G)n^{h-1}-d(G)>0$} and \[n^{h-1} \ge N+1 \ge  d(G)/\til{q}_h(G)+1.\] We have
\begin{align*}
    \til{q}_h(G) \le \frac{1}{|G|} \Big( |\Zen{G}| + \frac{k(G)-|\Zen{G}|}{n^{h-1}} \Big).
\end{align*}
Using that $d(G)=k(G)/|G|$ and rearranging, we find that
\begin{align}\label{eq:ch-inequality}
    [G:\Zen{G}] \leq  \frac{n^{h-1}-1}{\til{q}_h(G) n^{h-1}-d(G)}.
\end{align}
Since the right-hand side of Equation~(\ref{eq:ch-inequality}) is a decreasing function of $n^{h-1}$, we have
\begin{align*}
    [G:\Zen{G}] \leq  \frac{N}{\til{q}_h(G)(N+1)-d(G)} \le \frac{d(G)}{\til{q}_h(G)^2},
\end{align*}
as claimed.
\end{proof}


\section{Commutativity, nilpotency, supersolvability and solvability criteria.}\label{s:criteria}

In this section, we prove Theorems \ref{t:qh-criteria}(a)--(d). We begin with the commutativity criterion of Theorem \ref{t:qh-criteria}(a).


\vspace{1em}
\noindent \emph{Proof of Theorem \ref{t:qh-criteria}(a).}
Suppose for contradiction that $G$ is a counterexample to the claim with minimal order. The degrees of the complex irreducible characters of $\mr{D}_8$ are 1, 1, 1, 1 and~2. By a result due to Gustafson \cite{G73}, $d(G) \le 5/8$ since $G$ is nonabelian. Altogether, we have
\begin{align*}
    \frac12 < \frac12 \big(1+\frac{1}{2^{2h}}\big) = q_h(\mr{D}_8) < q_h(G) \le d(G) \le \frac58.
\end{align*}

We first claim that $G$ is nilpotent; that is, all Sylow subgroups of $G$ are normal in~$G$. By Lemma \ref{l:dG}(c), if a Sylow $p$-subgroup of $G$ is not normal for some prime $p$, then $d(G) \le 1/p$. But then $1/2 < 1/p$, which is impossible.

Next, we claim that all Sylow subgroups of odd order are abelian. Note that if \mbox{$G=A \times B$}, then $q_h(G) = q_h(A)q_h(B)$. Since $G$ is nilpotent by the last paragraph, this means that \mbox{$q_h(G)=\prod q_h(P)$}, where the product runs over the set of Sylow subgroups of~$G$. In particular, $q_h(P)>1/2$ for every Sylow subgroup $P$. However, by Lemma \ref{l:dG}(d), if $P$ is nonabelian then $d(P) < 1/p+1/p^2$. Thus, $1/p+1/p^2>1/2$. This inequality only holds if $p=2$, so all Sylow subgroups of odd order are abelian.

We may now assume that $G$ is a nonabelian $2$-group such that $q_h(G)>q_h(\mr{D}_8)$. Now, by Lemma \ref{l:dG}(e), either $d(G)<1/2$ or $|G'|=2$. We have $d(G)>1/2$ by the first paragraph, so $|G'|=2$. Therefore,
\begin{align*}
    q_h(\mr{D}_8) < q_h(G) 
    \le \frac12\Big(1-\frac{1}{2^{2h-2}} \Big) + \frac{d(G)}{2^{2h-2}},
\end{align*}
where we used that nonlinear irreducible characters of $G$ have degree at least $2$ and the observation in Remark~\ref{r:bound}. Thus,
\begin{align*}
    d(G) > \frac12 + 2^{2h-2}\Big(q_h(\mr{D}_8)-\frac12\Big) 
    = \frac12 + 2^{2h-2}\Big(\frac{1}{2^{2h+1}}\Big)
    = \frac12 + \frac18 = \frac58.
\end{align*}
But this contradicts the fact that $d(G) \le 5/8$ from the first paragraph, so the theorem is proved.
\qed

\vspace{1em}
Next, we prove the nilpotency criterion of Theorem \ref{t:qh-criteria}(b). We will show that a minimal counterexample is a so-called minimal non-nilpotent group; that is, a non-nilpotent group whose proper subgroups are nilpotent. The structure of such groups, which we summarize in the following lemma, is well understood:

\begin{lemma}{\cite[Theorem~III.5.2]{Hu67}}\label{l:nonnil}
Suppose $G$ is a minimal non-nilpotent group. Then $G \cong Q \rtimes P$, where $Q$ is a Sylow $q$-subgroup for some prime~$q$, 
while $P$ is a cyclic Sylow $p$-subgroup for some prime $p$ and $\Phi(P) \le \Zen{G}$.
\end{lemma}


\noindent \emph{Proof of Theorem \ref{t:qh-criteria}(b).}
Suppose for contradiction that $G$ is a counterexample to the claim with minimal order. The irreducible character degrees of $\mr{S}_3$ are 1, 1 and~2.
Since $G$ is not nilpotent, it contains a non-normal Sylow $p$-subgroup for some prime $p$ dividing~$|G|$. This means $d(G) \le 1/p \le 1/2$ by Lemma~\ref{l:dG}(c). Altogether, we have
\begin{align*}
    \frac13 < \frac13 \big(1+\frac{1}{2^{2h-1}}\big) = q_h(\mr{S}_3) < q_h(G) \le d(G) \le \frac1p \le \frac12.
\end{align*}
In particular, this means that the only non-normal Sylow subgroup is the Sylow $2$-subgroup. 

Now, by the monotonicity of Lemma~\ref{l:mono}, every proper subgroup of $G$ is nilpotent, so $G$ is a minimal non-nilpotent group by induction and has the structure described in Lemma~\ref{l:nonnil}. We know that $p=2$ by the last paragraph, and $q \neq 2$ since $G$ is not nilpotent. As $P$ is cyclic, $[P:\Phi(P)] = 2$.


Now consider the structure of $G'$: As $G$ is not nilpotent, of course $G'>1$. As \mbox{$d(G)>1/3$}, we have $|G'|<9$ by Lemma~\ref{l:G'}. Since $G/Q \cong P$ is abelian, $G' \le Q$. So $q=3$, $5$ or $7$ and $G'$ is cyclic of prime order $q$. 
Since $G' \nrm Q$, we have $G' \cap \Zen{Q} >1$, so $G' \le \Zen{Q}$ as $G'$ has prime order. So $Q$ centralizes $G'$. 
If $G'<Q$, then $G' P < G$ is nilpotent by induction on the group order, so $P$ centralizes $G'$, as well. But then $G' \le \Zen{G}$ and $G$ is nilpotent, a contradiction. So we have that $G'=Q$ is cyclic of order $3$, $5$ or $7$.

As $\Phi(P) \le \Zen{G}$, the product $Q \Phi(P)$ is abelian of index $2$ in $G$. So every nonlinear irreducible character of $G$ has degree $2$ since it is a constituent of a character induced from a (necessarily linear) character of $Q \Phi(P)$. As~$G$ has $[G:G']=|P|$ linear characters, the remaining nonlinear characters must be $(|Q|-1)|P|/4$ in number since the resulting sum of squared degrees is $|P| + \frac14 (|Q|-1)|P|\cdot 2^2 = |G|$. Therefore,
\begin{align*}
    q_h(G) = \frac{1}{|G|} \bigg( |P| + \frac{(|Q|-1)|P|}{4} \frac{1}{2^{2h-2}} \bigg) 
    = \frac{1}{q} \big( 1 + \frac{q-1}{2^{2h}} \big) > \frac13 \big(1 + \frac{2}{2^{2h}} \big) = q_h(\mr{S}_3).
\end{align*} 
But the inequality implies that $q<3$. We have reached our final contradiction and proved the claim.
\qed

\vspace{1em}

Now we prove the supersolvability criterion of Theorem \ref{t:qh-criteria}(c). We will show that a minimal counterexample is a so-called minimal non-supersolvable group; that is, a non-supersolvable group whose proper subgroups are supersolvable. The structure of such groups was determined by Doerk \cite{Doerk}, and can be read off from \cite[Theorem~12]{BE} and \cite[Chapter VI, Exercise~16]{Hu67}. We summarize the results we will initially need in the following lemma.

\begin{lemma}{\cite{Doerk}}\label{l:nonsol}
Suppose $G$ is a minimal non-supersolvable group. Then $G$ has exactly one normal Sylow $p$-subgroup $P$ and a complement $Q$ such that $|Q|$ is divisible by at most two distinct primes. Furthermore, $P/\Phi(P)$ is a minimal normal subgroup of $G/\Phi(P)$ and is not cyclic, and $Q\cap \Cen{G}{P/\Phi(P)}=Q\cap \Phi(G) = \Phi(Q) \cap \Phi(G)$.
\end{lemma}


\noindent \emph{Proof of Theorem \ref{t:qh-criteria}(c).}
Let $G$ be a counterexample to the theorem with minimal order. The irreducible character degrees of $\mr{A}_4$ are $1,1,1$ and $3$. Since $G$ is not supersolvable, $d(G) \le 1/3$ by \cite{BMN}. Altogether, we have
    \begin{align*}
    \frac14 < \frac14 (1+\frac{1}{3^{2h-1}}) = q_h(\mr{A}_4) < q_h(G) \le d(G) \le \frac13.
    \end{align*}
Then by Lemma \ref{l:mono}, for every proper subgroup $H$ of $G$, we have $q_h(H)\ge q_h(G)>q_h(\mr{A}_4)$ and hence by the minimality of $|G|$, $H$ is supersolvable. As $G$ is not supersolvable, $G$ is a minimal non-supersolvable group and has the structure described in Lemma~\ref{l:nonsol}. Now we work in $\q{G} = G/\Phi(G)$, which is again minimal non-supersolvable, since if $\q{G}$ is supersolvable then $G$ is supersolvable by \cite[Theorem~VI.8.5]{Hu67}. Recall that $\Phi(P) \le \Phi(G)$ by \cite[Lemma~III.3.3]{Hu67}; since $P/\Phi(P)$ is minimal normal in $G/\Phi(P)$ by Lemma~\ref{l:nonsol}, $\Phi(P)=P \cap \Phi(G)$. Therefore, $\q{P}$ is a minimal normal, elementary abelian, noncyclic subgroup of~$\q{G}$, where $\q{G}=\q{P}\q{Q}$ and $\q{Q}$ acts coprimely, faithfully and irreducibly on $\q{P}$. By Brauer's $k(GV)$-theorem (\cite[Theorem]{GMRS}), we know that $k(\q{G}) =k(\q{P} \rtimes \q{Q}) \le |\q{P}|$. Using this and Lemma~\ref{l:dG}(f), we have
\begin{align*}
    \frac14 < d(G) \le d(\q{G}) \le \frac{1}{|\q{Q}|},
\end{align*}
and so $|\q{Q}| \le 3$. Since $\q{Q}$ has prime order and $\q{P}$ is an irreducible $\q{Q}$-module, this means $\Cen{\q{G}}{\q{x}} \le \q{P}$ for all nontrivial $\q{x} \in \q{P}$, so $\q{G}$ is a Frobenius group. Now, if $\q{Q}=\gen{\q{t}}$ has order~$2$, then $\q{t}$ acts as inversion on~$\q{P}$ (see \cite[Lemma~7.21]{I94} and its proof). So every subgroup $\gen{\q{x}} \le \q{P}$ is fixed by $\q{Q}$, which means $\q{P}=\gen{\q{x}}$ is cyclic as it is an irreducible $\q{Q}$-module. This contradiction implies that $|\q{Q}|=3$. 

Now we count the characters of the Frobenius group $\q{G}$: Since $\q{G}'=\q{P}$, there are $[\q{G}:\q{G}']=|\q{Q}|$ linear characters, and each of the nonlinear irreducible characters of $\q{G}$ is induced from an irreducible character of $\q{P}$ lying in one of the $(|\q{P}|-1)/|\q{Q}|$ orbits under~$\q{Q}$. So
\begin{align*}
    \frac14 < d(\q{G}) = \frac{1}{|\q{P}||\q{Q}|}\bigg(|\q{Q}| + \frac{|\q{P}|-1}{|\q{Q}|}\bigg).
\end{align*}
Rearranging and using that $|\q{Q}|=3$, we find that $p^2 \le |\q{P}| < 32/5 < 7$. Therefore, $|\q{P}|=4$ and so $\q{G} \cong \mr{A}_4$. In particular, $P$ is a $2$-group.

Let $r$ be a prime divisor of $|G|$, and let $R$ be a Sylow $r$-subgroup of $G$. If $R$ is not normal in~$G$, then $d(G)\leq \frac1r.$ Since $d(G)\ge q_h(G) > q_h(\textrm{A}_4) > \frac14$ we deduce that $\frac1r>\frac14$ which forces $r<4$ and hence $r\in \{2,3\}$. In particular, $Q$ is a $\pi$-group, where $\pi\subseteq \{2,3\}$. Since we showed that the Sylow $p$-subgroup $P$ is a $2$-group, $Q$ must then be a $3$-group. Now, since $\Phi(Q) \ge \Phi(Q) \cap \Phi(G)=Q \cap \Cen{G}{P/\Phi(P)}$ and $\q{Q} \cong Q/(\Phi(Q) \cap \Phi(G))$ is cyclic of order $3$, we have $\Phi(Q)=\Phi(Q) \cap \Phi(G)$. So $Q$ is itself cyclic by Burnside's basis theorem \cite[Theorem~III.3.15]{Hu67}. We have shown that $G=PQ$, where $P$ is a $2$-group and $Q$ is a cyclic $3$-group, and $\q{G}=G/\Phi(G) \cong \mr{A}_4$.

Now we turn to the more precise classification of minimal non-supersolvable groups in \cite[Theorem~10]{BE}. Since $Q$ is cyclic of $3$-power order and $p=2$, inspecting the list of possibilities in \cite[Theorem~9]{BE} shows that only Types 2 and 3 are possible. We inspect each of these in turn and derive contradictions:

Type 2: \emph{$Q=\gen{z}$ is cyclic, $P$ is an irreducible $Q$-module over the field of $p$ elements with kernel $\gen{z^q}$ in $Q$.}

Since $P \nrm G$ is an irreducible $Q$-module, $P$ has no $Q$-invariant nontrivial proper subgroup. In particular, $P \cap \Phi(G) = 1$ and so $\q{P} \cong P$ is elementary abelian of order~$4$ by what we showed above. The kernel $K=\gen{z^3}$ is central in $G$ and has index $3$ in~$Q$, so $|Q|=3|K|$ and $|G|=12|K|$. We have $[G:G']=[G:P]=|Q|=3|K|$, and the nonlinear irreducible characters of $G$ are induced from characters $\psi_1 \times \psi_2 \in \Irr{P \times K}$ with $\psi_1 \neq 1$, which split into $|K|$ orbits under~$G$. Thus, each of these $|K|$ nonlinear characters of $G$ has degree $3$. So 
\begin{align*}
    q_h(G) = \frac{1}{12|K|}\bigg(3|K| + \frac{|K|}{3^{2h-2}} \bigg) = \frac14 \bigg(1 + \frac{1}{3^{2h-1}} \bigg) = q_h(\mr{A}_4),
\end{align*}
which is a contradiction. Note that if $K=1$, then $G \cong \mr{A}_4$.

Type 3: \emph{$P$ is a nonabelian special $p$-group of rank $2m$, the order of $p$ modulo $q$ being~$2m$, $Q=\gen{z}$ is cyclic, $z$ induces an automorphism on $P$ such that $P/\Phi(P)$ is a faithful and irreducible $Q$-module, and $z$ centralizes $\Phi(P)$. Furthermore, $|P/\Phi(P)|=p^{2m}$ and $|P'|\le p^m$.}

Since $|\q{P}|=4$, we have $m=1$. So $|P'| \le 2$, which means $P$ is an extraspecial $2$-group of order~$2^3$. Since $P/\Phi(P)$ is a faithful $Q$-module, $Q$ is cyclic of order $3$. If $P$ is an extraspecial $2$-group of order $2^3$, then $P$ has $4$ linear characters and $1$ nonlinear irreducible character of degree $2$. The nontrivial linear characters form one $Q$-orbit and induce irreducibly to the same nonlinear character of $G$ of degree $3$. The unique nonlinear character of $P$ is necessarily $G$-invariant, and so extends to $G$ since $[G:P]=3$ is prime (see \cite[Corollary~6.19]{I94}). This gives rise to $|\Irr{G/P}|=3$ distinct irreducible characters of $G$ of degree $2$ by \cite[Corollary~6.17]{I94}. There are $[G:G']=[G:P]=3$ linear characters. So
\begin{align*}
    q_h(G) = \frac{1}{24}\bigg( 3 + \frac{3}{2^{2h-2}} + \frac{1}{3^{2h-2}} \bigg) \le \frac{1}{24} \bigg(6 + \frac{1}{3^{2h-2}} \bigg) = \frac{1}{4} \bigg(1 + \frac{1}{2 \cdot 3^{2h-1}} \bigg) < q_h(\mr{A}_4),
\end{align*}
a contradiction. This is our final contradiction and the theorem is proved.
\qed

\vspace{1em}
Now we prove the solvability criterion of Theorem~\ref{t:qh-criteria}(d).
\vspace{1em}


\noindent \emph{Proof of Theorem \ref{t:qh-criteria}(d).} If $h=1$, then $q_1(G) = d(G) > d(\mr{A}_5) = 1/12$, and $G$ is already known to be solvable by a result of Dixon \cite{D73}. So, now assume that $h \ge 2$, and suppose that $G$ is a counterexample to the claim with minimal order. By Lemma \ref{l:dG}(a), we have \mbox{$d(G) \ge q_h(G) > q_h(\mr{A}_5) > 1/|\mr{A}_5|=1/60$}.

We first note that $q_h(G') \ge q_h(G)$ by Lemma \ref{l:mono}. So if $G' <G$, then $G'$ is solvable by induction on the order of the group. But then $G$ is also solvable, a contradiction. So $G=G'$ is perfect and nontrivial irreducible characters have degree at least $2$. Thus,
\begin{align*}
    q_h(G) 
    \le \frac{1}{|G|} \bigg(1+\frac{k(G)-1}{2^{2h-2}} \bigg) 
    =\frac{1}{|G|} \bigg( \frac{k(G)+2^{2h-2}-1}{2^{2h-2}}\bigg).
\end{align*}

Now note that we have the following inequalities when $k(G) \ge 5$:
\begin{align*}
    \frac{k(G)+2^{2h-2}-1}{2^{2h-2}} &\le \frac{2 k(G)}{5} \qquad \text{for $h=2$}, \\
    \frac{k(G)+2^{2h-2}-1}{2^{2h-2}} &\le \frac{k(G)}{3} \qquad \text{for $h>2$.}
\end{align*}
It is not difficult to see that if $k(G) < 5$, then $|G| \le 12$, and so $G$ is solvable (see \cite[Note~A]{B55} or \cite[Table 1]{L85}). Since $G$ is nonsolvable, we must have $k(G) \ge 5$. Therefore, we have
\begin{align*}
    d(G) = \frac{k(G)}{|G|} &\ge \frac{5}{2} q_2(G) > \frac{5}{2} q_2(\mr{A}_5) = \frac{5}{2} \frac{4769}{216000} > \frac{1}{20}, \\
    d(G) = \frac{k(G)}{|G|} &\ge 3 q_h(G) > 3 q_h(\mr{A}_5) > \frac{3}{60} = \frac{1}{20} \qquad \text{for $h >2$}.
\end{align*}
Therefore, $G$ is perfect with $d(G)>1/20$ for all $h \ge 2$. Then by Lemma \ref{l:dG}(b), $G \cong \mr{A}_5$ or $G \cong \mr{SL}_2(5)$. But
\begin{align*}
    q_h(\mr{SL}_2(5)) &= \frac{1}{120} \bigg(1 + \frac{2}{2^{2h-2}} + \frac{2}{3^{2h-2}} + \frac{2}{4^{2h-2}}+ \frac{1}{5^{2h-2}} + \frac{1}{6^{2h-2}} \bigg) \\
    &< \frac{1}{120} \bigg(2 + \frac{4}{3^{2h-2}} + \frac{2}{4^{2h-2}} + \frac{2}{5^{2h-2}} \bigg) =q_h(\mr{A}_5).
\end{align*}
This is our final contradiction, and the theorem is proved. \qed

\vspace{1em}

Next, we prove Theorem \ref{t:plocal} in a manner very similar to the proof of Theorem \ref{t:qh-criteria}(d). Although it may be true that $q_{h,p'}(G)$ is monotone with respect to subgroups, the proof of Lemma~\ref{l:mono} will not work when the prime $p$ divides $|G|$ as Frobenius reciprocity does not hold in this case. However, we can again show that a minimal counterexample to the theorem must be perfect by using some basic results from the theory of Brauer characters. Before we begin, recall that we defined $\alpha(h,p) = (2^{2h-2}+\sqrt{p-1})/(2^{2h-2}\sqrt{p-1})$.


\vspace{1em}
\noindent \emph{Proof of Theorem \ref{t:plocal}.} Suppose for contradiction that $G$ is a counterexample to the claim with minimal order. We first show that if $N \le G$ is a normal subgroup with index equal to any prime $r$, then $q_{h,p'}(G) \le q_{h,p'}(N)$. This will then imply that $G$ is perfect, since if not, then there exists such an $N \nrm G$ with $\alpha(h,p)/(p-1) < q_{h,p'}(G) \le q_{h,p'}(N)$, so $N$ is $p$-solvable by induction on the order of the group. But then $G$ is also $p$-solvable since the quotient $G/N$ is cyclic, which is a contradiction. 

Now, each character $\varphi \in \IBr{G}$ lies above at least one character $\theta \in \IBr{N}$. Denoting the collection of irreducible $p$-Brauer characters of $G$ lying over $\theta \in \IBr{N}$ by $\IBr{G|\theta}$, we have
\begin{align*}
    q_{h,p'}(G) \le \frac{1}{|G|_{p'}} \sum_{\theta \in \IBr{N}} \sum_{\varphi \in \IBr{G|\theta}} \frac{1}{\varphi(1)^{2h-2}}.
\end{align*}
Since $N \nrm G$ has prime index $r$, there are two possibilities for each $\theta \in \IBr{N}$: 

(1) $\theta$ is $G$-invariant, in which case $\theta$ extends to an irreducible Brauer character $\hat{\theta}$ of $G$ by \cite[Theorem~8.12]{N98} since $G/N$ is cyclic. In addition, the characters $\beta \hat{\theta}$ for $\beta \in \IBr{G/N}$ are all the irreducible constituents of $\theta^G$ by \cite[Corollary~8.20]{N98}. Note that $\beta(1)\hat{\theta}(1)=\theta(1)$ since $\beta$ is linear. Finally, since $N \nrm G$, $\varphi \in \IBr{G}$ is an irreducible constituent of $\theta^G$ if and only if $\theta$ is an irreducible constituent of $\varphi_N$ by \cite[Corollary~8.7]{N98}. Therefore, \mbox{$|\IBr{G|\theta}|=|\IBr{G/N}|=[G:N]_{p'}$}. 

(2) The stabilizer of $\theta$ in $G$ is $N$, in which case the induced character $\varphi = \theta^G$ is irreducible by Clifford's theorem \cite[Theorem 8.9]{N98}. So $|\IBr{G|\theta}|=1$, again by \cite[Corollary~8.7]{N98}, and $\varphi(1) > \theta(1)$. 

Therefore, for each $\theta \in \IBr{N}$ we have
\begin{align*}
    \sum_{\varphi \in \IBr{G|\theta}} \frac{1}{\varphi(1)^{2h-2}} \le \frac{[G:N]_{p'}}{\theta(1)^{2h-2}},
\end{align*}
and so
\begin{align*}
    q_{h,p'}(G) \le \frac{[G:N]_{p'}}{|G|_{p'}} \sum_{\theta \in \IBr{N}} \frac{1}{\theta(1)^{2h-2}}  = q_{h,p'}(N),
\end{align*}
which is what we wanted to show.

Since we now know that the counterexample $G$ is perfect, all nontrivial irreducible Brauer characters are nonlinear, and we have
\begin{align*}
    q_{h,p'}(G) 
    \le \frac{1}{|G|_{p'}} \bigg(1+\frac{k_{p'}(G)-1}{2^{2h-2}} \bigg) 
    =\frac{1}{|G|_{p'}} \bigg( \frac{k_{p'}(G)+2^{2h-2}-1}{2^{2h-2}}\bigg),
\end{align*}
where the inequality follows since $\varphi(1) \ge 2$ for all nonlinear $\varphi \in \IBr{G}$.

Now, since $G$ is non-$p$-solvable, it is known that $k_{p'}(G) > \sqrt{p-1}$ by \cite[Theorem~1.2]{NM22}. Note that the following inequality is satisfied for $k_{p'}(G) > \sqrt{p-1}$ and each $h \ge 1$:
\begin{align*}
    \frac{k_{p'}(G)+2^{2h-2}-1}{2^{2h-2}} < \alpha(h,p) k_{p'}(G).
\end{align*}
To see this, multiply through by $2^{2h-2}$, move all terms involving $k_{p'}(G)$ to the right-hand side, and use that $k_{p'}(G) > \sqrt{p-1}$. This implies that
\begin{align*}
    d_{p'}(G) = \frac{k_{p'}(G)}{|G|_{p'}} > \frac{1}{\alpha(h,p)} q_{h,p'}(G) > \frac{1}{p-1}.
\end{align*}
But this contradicts \cite[Theorem~1.1]{S24}, and the theorem is proved.
\qed

\section{Criterion for the existence of normal Sylow subgroups}\label{sec:p-closed}

In this section, we prove Theorem~\ref{t:pclosed}, which is a criterion for the existence of a normal Sylow subgroup. Let $p$ be a prime, and let $h$ be a positive integer. Recall that for a finite group $G$ and a prime~$p$, $G$ is said to be $p$-closed if $G$ has a normal Sylow $p$-subgroup. Define 
\begin{align*}
    \gamma(h,p)=\frac{1}{p+1}\Big(1+\frac{1}{p^{2h-1}}\Big).
\end{align*}
Note that $\gamma(h,p)>{1}/{(p+1)}$. We now prove Theorem~\ref{t:pclosed}.

\medskip
\emph{Proof of Theorem~\ref{t:pclosed}.} Let $G$ be a counterexample to the theorem with minimal order. Then $q_h(G)>\gamma(h,p)$ but $G$ does not have a normal Sylow $p$-subgroup. By Lemma~\ref{l:dG}(c), we have $d(G)\leq 1/p.$ Consequently, we obtain the following inequalities:   
\begin{equation*}\label{eqn1} 
\frac{1}{p}\ge d(G)\ge q_h(G)>\gamma(h,p)>\frac{1}{p+1}.
\end{equation*}
By Lemma \ref{l:mono}, for each proper subgroup $H$ of $G$, we have $q_h(H)\ge q_h(G)>\gamma(h,p)$ and thus by the minimality of $|G|$, $H$ is $p$-closed. In particular, every proper subgroup of $G$ is $p$-closed. Also, if $N$ is a normal subgroup of $G$, then every proper subgroup of the quotient group $G/N$ is $p$-closed. 

Assume that $p=2$. Then $\gamma(h,p)=q_h(\textrm{S}_3)$ and since $q_h(G)>q_h(\textrm{S}_3)$, $G$ is nilpotent by Theorem~\ref{t:qh-criteria}(b). But then $G$ has a normal Sylow $2$-subgroup, a contradiction. Therefore, we may assume that $p\ge 3.$ Let $P$ be a Sylow $p$-subgroup of $G$.

\medskip
\textbf{Claim $1$}: \emph{$\langle P^G\rangle=G$. Hence, $G$ has no proper normal subgroup of $p'$-index.}

Let $L=\langle P^G\rangle.$  Then $P\leq L\unlhd G.$ Suppose that $L<G$. By the minimality of $|G|$, $L$ has a normal Sylow $p$-subgroup, which implies that $P\unlhd L.$ Since $P$ is characteristic in $L$, we deduce that $P\unlhd G$, which is a contradiction. Therefore, $G=L$ as desired.

\medskip
\textbf{Claim $2$}: \emph{$G$ is not a nonabelian simple group.}  

Since $d(G)>1/(p+1)$, the result follows from Lemma \ref{lem1}.

\medskip
\textbf{Claim $3$}: \emph{Let $N$ be a maximal normal subgroup of $G$. Then $G=\langle x\rangle N$ with $[G:N]=p$, where $N$ has a normal Sylow $p$-subgroup, which is $O_p(G)$, and $x\in P$ is a $p$-element and $P=\langle x\rangle O_p(G)$. In particular, $G$ is $p$-solvable with a Hall $p'$-subgroup $R$ and $G=PR$ with $[P:O_p(G)]=p$.}

Since $G$ is not a simple group, $G$ has a maximal normal subgroup, say $N$. It follows that $N$ has a normal Sylow $p$-subgroup, say $P_1$. As $P_1\unlhd N\unlhd G$, we deduce that $P_1\unlhd G$ and thus $P_1\leq O_p(G)$. Since $G/N$ is a simple group and $d(G/N)\ge d(G)>1/(p+1)$ and $p$ divides $|G/N|$ by Claim $1$, Lemma \ref{lem1} yields that $G/N$ is abelian. Hence, $G/N$ is cyclic of order~$p$. 
It follows that $G=PN$ and $P\cap N=P_1\unlhd G.$ As $[G:N]=[P:P_1]=p$ and $P$ is not normal in $G$, we deduce that $P_1=O_p(G)$. Let $x\in P\setminus P_1$. Then $P=\langle x\rangle P_1=\langle x\rangle  O_p(G)$. Since $N/O_p(G)$ is a $p'$-group and both $O_p(G)$ and $G/N$ are $p$-groups, $G$ is $p$-solvable. So $G$ has a Hall $p'$-subgroup $R$ and thus $N=RO_p(G)$ and $G=PN=PR$ with $[G:N]=[P:O_p(G)]=p.$

\medskip
\textbf{Claim $4$}: \emph{$R$ is a Sylow $r$-group of $G$ for some prime $r$.}

Let $\overline{G}=G/O_p(G)$. Then $\overline{P}\cong C_p$ acts coprimely and nontrivially on $\overline{N} = \overline{R}\cong R$. Suppose by contradiction that $|\overline{R}|=|R|$ is divisible by at least two distinct primes. Let $q$ be a prime divisor of $|R|$. By a coprime action theorem (\cite[Theorem~3.23]{I08}), $\overline{P}$ stabilizes a Sylow $q$-subgroup $\overline{Q}$ of~$\overline{R}$. But then $\overline{P}\overline{Q}$ is a proper subgroup of~$\overline{G}$, and hence it has a normal Sylow $p$-subgroup. Thus, $\overline{P}$ centralizes~$\overline{Q}$. 
It follows that $|\overline{G}:\Cen{\overline{G}}{\overline{P}}|$ is not divisible by $q$. Clearly, this index is also not divisible by $p$. As this is true for all prime divisors of $|R|$, we must have $\overline{G}=\Cen{\overline{G}}{\overline{P}}$, which is a contradiction. Thus $R$ is an $r$-group for some prime $r$, and so it is a Sylow $r$-subgroup of $G$.

\medskip
\textbf{Claim $5$}: \emph{Let $n$ be the smallest nontrivial character degree of $G$. Then $n\ge (p-1)/2$.}

Note that $G$ is nonabelian and so such a minimal nontrivial degree $n$ exists. Let $\chi\in\textrm{Irr}(G)$ such that $\chi(1)=n$ and let $K$ be the kernel of $\chi.$ Then $K$ is a proper  normal subgroup of $G$. We consider the following cases.

Case 1: $G/K$ has a normal Sylow $p$-subgroup.  Then $PK/K\unlhd G/K$ as $PK/K$ is a Sylow $p$-subgroup of $G/K$. By Claim $1$,  we have $G=PK$ and thus $G/K\cong P/P\cap K$, which is a $p$-group. Since $\chi\in\Irr {G/K}$
is nonlinear, $G/K$ is a nonabelian $p$-group and thus $n=\chi(1)\ge p>(p-1)/2$.

Case 2: $G/K$ does not have a normal Sylow $p$-subgroup.
It follows that $G/K$ has a faithful irreducible character of degree $n$ and does not have a normal Sylow $p$-subgroup. By \cite[Theorem~1]{FT}, $p\leq 2n+1$ or $n\ge (p-1)/2$.

\medskip
\textbf{Claim $6$}: \emph{$G'=RU$, where $U=G'\cap O_p(G)\unlhd G$.}

Note that $G$ is a $\{p,r\}$-group by Claim $4$, so it is solvable and $G'\leq N\unlhd G.$ Since $\langle P^G\rangle=G$, we deduce that $G/G'$ is an abelian $p$-group. Thus, $G'$ contains a Sylow $r$-group of $G$. It follows that $R\leq G'.$ By Dedekind's modular law, $G'=R(G'\cap O_p(G))$, where $G'\cap O_p(G)$ is a normal Sylow $p$-subgroup of $G'$.

\medskip
\textbf{Claim $7$}: \emph{The quotient $G/R'O_p(G)$ is a Frobenius group with a cyclic complement of order~$p$ and Frobenius kernel an 
abelian $2$-group of order $2^m$ for some $m\ge 2$, and \mbox{$p=2^m-1$.}}

Note that $N/O_p(G)=RO_p(G)/O_p(G)\unlhd G/O_p(G)$ and thus $(N/O_p(G))' = R'O_p(G)/O_p(G)$. Hence, $R'O_p(G)\unlhd G.$  Let $\overline{G}=G/R'O_p(G)$.
Then $O_p(\overline{G})=1$, $\overline{R}\unlhd \overline{G}$, $\overline{G}'=\overline{R}$ and $\overline{G}=\langle \overline{x}\rangle \overline{R}$
with $|\overline{x}|=p.$  Moreover, $\overline{R}$ is an  abelian $r$-group and $|\overline{G}'|=|\overline{R}|=r^m$ for some positive integer~$m$. It follows that $p$ is the only nontrivial character degree of $G$. Hence,
\begin{align*}
    d(\overline{G}) = \frac{1}{p^2}+(1-\frac{1}{p^2})\frac{1}{|\overline{G}'|}
\end{align*}
by \cite[Lemma~2(vi)]{G06}. Since $d(\overline{G})>1/(p+1)$, we deduce that \[r^m<\frac{(p+1)(p^2-1)}{p^2-p-1}.\] Since $p\ge 3,$ we can check that \[\frac{(p+1)(p^2-1)}{p^2-p-1}\leq 2p+1.\] Hence $r^m\leq 2p.$ As $r\neq p$, we have $r^m\leq 2p-1$. Moreover, as \[\frac{1}{|\overline{G}'|}\leq d(\overline{G})\leq  \frac{1}{p},\] we deduce that $r^m\ge p.$
Thus \begin{equation}\label{eqn2}p\leq r^m\leq 2p-1.\end{equation}

By Fitting's theorem \cite[Theorem~4.34]{I08}, $\overline{R}=[\overline{R},\overline{P}]\times \Cen{\overline{R}}{\overline{P}}$. If $\Cen{\overline{R}}{\overline{P}}$ is nontrivial, then $[\overline{R},\overline{P}] \overline{P}$ is a proper subgroup of $\overline{G}$, which implies that it is $p$-closed and hence $\overline{P}$ centralizes $[\overline{R},\overline{P}]$. This is impossible as it would imply that $\overline{P}$ centralizes $\overline{R}$ and thus $\overline{G}$ has a normal Sylow $p$-subgroup. Therefore, $\Cen{\overline{R}}{\overline{P}}=1.$ It follows that $\overline{G}$ is a Frobenius group with kernel $\overline{R}$ and a cyclic complement $\overline{P}$. This implies that $r^m-1=kp$ for some positive integer $k$. Combining with Equation \eqref{eqn2}, we get $r^m-1=p$. As $p\ge 3$, $p+1=r^m$ must be even. Hence $r=2$ and 
\begin{equation*}\label{eqn3}
2^m-1=p.
\end{equation*} In particular, $m\ge 2$ is a prime and $p$ is a Mersenne prime. 

For brevity, we define $\Gamma(p,m):= C_2^m \rtimes C_p$ to be a Frobenius group with kernel an elementary abelian $2$-group of order $2^m$ and cyclic complement of order $p$ with $2^m-1=p$. 


\medskip
\textbf{Claim $8$}: \emph{We have $R'=1$, so $G/O_p(G) = G/R'O_p(G) \cong \Gamma(p,m)$ with $p=2^m-1\ge 7$ and $m\ge 3$ an odd prime.}

Assume that $m=2$. Then $p=3$ and
\begin{align*}
    \gamma(h,3)=\frac{1}{4}(1+\frac{1}{3^{2h-1}})=q_h(\textrm{A}_4).
\end{align*} 
Hence, $q_h(G)>q_h(\textrm{A}_4)$ and so by Theorem \ref{t:qh-criteria}(c), $G$ is supersolvable. By \cite[Theorem~VI.9.1(c)]{Hu67}, as $G$ is a $\{2,3\}$-group, $G$ has a normal Sylow $3$-subgroup, which is a contradiction.
Therefore, we may assume from now on that $m\ge 3$ and $p=2^m-1\ge 7.$

Working in the quotient group $G/O_p(G)$, we may assume $O_p(G)=1.$ By Claim 6, $G'=R$ and so $G''=R'$ with $[R:R']=2^m$ by Claim 7.
Recall that $n$ is the smallest nontrivial character degree of $G$. Assume first that $n\ge p.$  By \cite[Lemma 2(vi)]{G06}, we have 
\begin{align*}
    \frac{1}{p+1}<d(G)\leq \frac{1}{n^2}+(1-\frac{1}{n^2})\frac{1}{|G'|}\leq \frac{1}{p^2}+(1-\frac{1}{p^2})\frac{1}{|G'|}.
\end{align*}
As above, we deduce that $p\leq |G'|\leq 2p-1$. Write $|G'|=|R|=2^{m+c}$ for some integer $c\ge 0.$ Recall that $2^m=1+p$ so 
\begin{align*}
    p\leq |G'|=(1+p)2^c=2^cp+2^c\leq 2p-1.
\end{align*}
This implies that $2^c=1$, and hence $G'=R$ is abelian.

Now assume that $n<p$. As $|G|=2^{m+c}p$ and $n$ divides $|G|$ and is coprime to $p$, we deduce that $n=2^a$ for some positive integer $a\ge 1$. Note that $ (p-1)/2\leq 2^a$ by Claim~5 and thus 
\begin{align*}
    \frac{p-1}{2}\leq 2^a<p=2^m-1.
\end{align*}
It follows that \[\frac{1}{2}(2^m-2)\leq 2^a\leq p-1=2^m-2.\] 
Therefore, 
\begin{align*}
    2^{m-1}-1=\frac{1}{2}(2^m-2)\leq 2^a\leq 2^m-2.
\end{align*}
We deduce that $n=2^a=2^{m-1}.$
By \cite[Lemma 2 (vi)]{G06}, we have 
\begin{align*}
    \frac{1}{p+1}=\frac{1}{2^m}<d(G)\leq \frac{1}{2^{2(m-1)}}+(1-\frac{1}{2^{2(m-1)}})\frac{1}{|G'|}.
\end{align*} 
Since $m\ge 3$, simplifying the previous inequality, we obtain \[|G'|<2^m+\frac{2^m-1}{2^{m-2}-1}.\] Now, the second summand on the right hand side of the inequality is at most $7$ and thus $2^{m+c}=|G'|< 2^m+7.$ However, this is impossible if $c\ge 1$ as $m \ge 3$.

We showed in both cases that $G'=R$ is abelian. If $R$ is not elementary abelian, then the subgroup $S$ generated by all elements of $R$ of order $2$ is a proper subgroup of $R$ stabilized by~$P$, so $SP<G$ has a normal Sylow $p$-subgroup. Then $P \nrm SP$ fixes every element of order $2$ in $R$, and so $P$ acts trivially on $R$ by \cite[Corollary~4.35]{I08}. But then $P \nrm G$, a contradiction. Thus, $R$ is elementary abelian and $G\cong \Gamma(p,m)$, as claimed.  

\medskip
\textbf{Claim $9$}: \emph{The final contradiction.}
 
Assume $O_p(G)=1$. Then $G\cong\Gamma(p,m)$, and so $[G:G']=p$ and $p$ is the unique nontrivial character degree of $G$. Hence, $q_h(G)=\gamma(h,p)$, which is a contradiction. 
 
Assume that $O_p(G)>1$.  Note that $G'=RU$ and $U=G'\cap O_p(G)$ is a normal Sylow $p$-subgroup of $G'$. 
 
(i) Assume that $p\nmid |G'|$. So $G'=R$ is a normal Sylow $r$-subgroup of $G$. Now $P=\langle x\rangle O_p(G)$ is abelian.
In this case, $G'\langle x\rangle\unlhd G.$ If $G'\langle x\rangle <G$, then $x$ centralizes $G'$, which is impossible. Thus $G=G'\langle x\rangle.$ In particular, $P=\langle x\rangle$ and $y=x^p$ centralizes $G'=R.$
 
By Claim 8, $G/O_p(G)\cong \Gamma(p,m)$ is a Frobenius group with $G'=R$ an elementary abelian subgroup of order $2^m=p+1$. Thus $\mbf{F}(G)=G'\times \langle y\rangle=G'\times O_p(G)$ is abelian and in fact $O_p(G)=\langle y\rangle$ is central in $G$. It follows that $p$ is the only nontrivial character degree of $G$. Since $[G:G']=|P|$, we have  \[q_h(G)=\frac{1}{|G|}(|P|+\frac{k(G)-|P|}{p^{2h-2}}) \text{ and } |G|=[G:G']+p^2(k(G)-|P|).\] Solving for $k(G)-|P|$ from the second equation and plugging it in the former, we get  $q_h(G)=\gamma(h,p)$, a contradiction.
 
 \medskip
 (ii) Assume $p\mid |G'|$. So $|G'|=|R|p^e$ for some positive integer $e.$ Let $n$ be the minimal nontrivial degree of $G$. By \cite[Lemma 2 (vi)]{G06}, we have
\begin{equation}\label{eqn4}\frac{1}{p+1}<\frac{1}{n^2}+(1-\frac{1}{n^2})\frac{1}{|G'|}.\end{equation}
 
Assume first that $n\ge p$. Then Equation \eqref{eqn4} implies that \[|G'|<\frac{(p+1)(p^2-1)}{p^2-p-1}.\] As in the proof of Claim 8, we deduce that $|G'|\leq 2p-1$. However, since $|G'|=|R|p^e$ with $e\ge 1$ and $|R|= 2^m=1+p\ge 8$, we see that $|G'|>2p-1,$ which is a contradiction.
 
Assume that $n<p.$ By Claim 5, we have $n\ge (p-1)/2$. Thus $(p-1)/2\leq n\leq p-1$. Let $\chi\in \textrm{Irr}(G)$ with $\chi(1)=n.$ Recall that $2^m=p+1$ and $G/O_p(G)\cong\Gamma(p,m)$. It follows that $\chi(1)$ divides $|R|$ where $|R|=2^m$ since $p\nmid \chi(1)$ and $\chi(1)\mid |G|$. So $n=2^a$ for some positive integer $a$. Since $(p-1)/2=(2^m-2)/2=2^{m-1}-1<2^a\leq p-1=2^m-2$, we deduce that $n=2^{m-1}$.
 
Employing the same argument as in the proof of Claim 8, we obtain a contradiction as $m\ge 3.$
 
The proof of the theorem is now complete.
\qed

\section{The dual invariant}\label{s:ch-criteria}

In this section, we will prove Theorems \ref{t:ch-criteria} and \ref{t:pclosed-dual}. 
First, we prove the abelian criterion of Theorem~\ref{t:ch-criteria}(a).

\vspace{1em}
\noindent\emph{Proof of Theorem~\ref{t:ch-criteria}(a).}
Assume by contradiction that $G$ is a counterexample to the claim with minimal order. Note that the sizes of the conjugacy classes of $\mr{D}_8$ are $1,1,2,2,2$. Since $\til{q}_1(G)=d(G)$, we must have $h\ge 2$. As $G$ is nonabelian, we have $d(G) \le 5/8$ by~\cite{G73}. Every noncentral conjugacy class of $G$ has size at least $2$, so we have
\begin{align*}
    \frac14 \Big(1 + \frac{3}{2^{h}} \Big) &= \til{q}_h(\mr{D}_8) < \til{q}_h(G) \le \frac{1}{|G|}\Big( |\Zen{G}| + \frac{k(G)-|\Zen{G}|}{2^{h-1}} \Big) \\
    &= \frac{1}{[G:\Zen{G}]}\Big(1-\frac{1}{2^{h-1}}\Big) + \frac{d(G)}{2^{h-1}} \\
    &\le \frac{1}{[G:\Zen{G}]}\Big(1-\frac{1}{2^{h-1}}\Big) + \Big(\frac{1}{4} \Big) \Big( \frac{5}{2^{h}}\Big).
\end{align*}
We find immediately from these inequalities that $[G:\Zen{G}]<4$. But then $G/\Zen{G}$ is cyclic and $G$ is abelian, a contradiction.
\qed

\vspace{1em}

Next, we prove the nilpotency criterion of Theorem~\ref{t:ch-criteria}(b).

\vspace{1em}
\noindent\emph{Proof of Theorem~\ref{t:ch-criteria}(b).} Assume by contradiction that $G$ is a counterexample to the claim with minimal order. Since $\til{q}_1(G)=d(G)$, we must have $h\ge 2$. We have
\begin{align*}
    \frac16 < \frac16 \Big(1 + \frac{1}{2^{h-1}} + \frac{1}{3^{h-1}} \Big) = \til{q}_h(\mr{S}_3) < \til{q}_h(G) \le d(G) \le \frac12.
\end{align*}

As $\til{q}_h(G/\Zen{G}) \ge \til{q}_h(G)$ by Lemma~\ref{lem:quotient}, the quotient $G/\Zen{G}$ is nilpotent if $\Zen{G}>1$ by the minimality of $|G|$. But then $G$ is nilpotent, a contradiction. Therefore, we may assume that $|\Zen{G}|=1$.

Since $d(G)/\til{q}_h(G) \le (1/2)/(1/6)=3$, we may take $N=3$ in Lemma~\ref{l:c-inequality}. Letting $n \ge 2$ be the minimal nontrivial class size, if $n^{h-1} \le N=3$, then we must have $h=2$.  In this case,
\begin{align*}
    \frac{11}{36} = \til{q}_2(\mr{S}_3) < \til{q}_2(G) \le \frac{1}{|G|}\Big( 1 + \frac{k(G)-1}{2} \Big) = \frac{1}{2|G|} + \frac{d(G)}{2} \le \frac{1}{2|G|} + \frac14,
\end{align*}
which means $|G|<9$ and so $G \cong \mr{S}_3$, a contradiction. Therefore, it only remains to consider case (2) of Lemma~\ref{l:c-inequality}; that is, $|G| \leq d(G)/\til{q}_h(G)^2 < 18$. The only nonnilpotent groups with trivial center of order less than $18$ are dihedral groups $\mr{D}_{2k}$ for $k=3,5,7$ and~$\mr{A}_4$. Since
\begin{align*}
    \til{q}_h(\mr{D}_{2k}) = \frac{1}{2k}\Big(1+\frac{(k-1)/2}{2^{h-1}} + \frac{1}{k^{h-1}} \Big) 
\end{align*}
when $k$ is odd, $\til{q}_h(\mr{D}_{2k})$ is a decreasing function of $k$, and so $\til{q}_h(G) \le \til{q}_h(\mr{S}_3)$ for all $h$ in these cases. The class sizes of $\mr{A}_4$ are $1,3,4,4$, so
\begin{align*}
    \til{q}_h(\mr{A}_4) = \frac{1}{12} \Big(1 + \frac{1}{3^{h-1}}+\frac{2}{4^{h-1}} \Big) < \frac{1}{12} \Big(2 + \frac{2}{3^{h-1}}+\frac{2}{4^{h-1}} \Big) = \frac{1}{6} \Big(1 + \frac{1}{3^{h-1}} + \frac{1}{4^{h-1}}\Big) < \til{q}_h(\mr{S}_3),
\end{align*}
a contradiction. The theorem is proved.
\qed
\vspace{1em}

We now prove the supersolvability criterion of Theorem~\ref{t:ch-criteria}(c).

\vspace{1em}
\noindent\emph{Proof of Theorem~\ref{t:ch-criteria}(c).}
Assume by contradiction that $G$ is a counterexample to the claim with minimal order. Since $\til{q}_1(G)=d(G)$, we must have $h\ge 2$. We have
\begin{align*}
    \frac{1}{12} < \frac{1}{12} \Big(1 + \frac{1}{3^{h-1}} + \frac{2}{4^{h-1}} \Big)= \til{q}_h(\mr{A}_4) < \til{q}_h(G) \le d(G) \le \frac{1}{3}.
\end{align*}
Since $d(G)>1/12$, $G$ is solvable by \cite{D73}. Since $G$ is supersolvable if $G/\Phi(G)$ is supersolvable by \cite[Theorem~VI.8.6(a)]{Hu67} and $G$ is supersolvable if $G/\Zen{G}$ is supersolvable by the definition of supersolvability \cite[Definition~VI.8.5(b)]{Hu67}, it follows by Lemma~\ref{lem:quotient} and induction that both the Frattini subgroup and the center of $G$ are trivial. 

Suppose that $G$ has two distinct minimal normal subgroups, say $M_1$ and $M_2$. By Lemma \ref{lem:quotient} and the minimality of $|G|$, both $G/M_1$ and $G/M_2$ are supersolvable, and so $G$ is supersolvable as $G$ embeds into a supersolvable group $G/M_1\times G/M_2$. This is a contradiction. Therefore, $G$ has a unique minimal normal subgroup, say $M$, and since $G$ is solvable, $M$ is an elementary abelian $p$-group of order $p^m$ for some prime $p$ and positive integer $m\ge 1$. If $m=1$, then $G$ is supersolvable, so we may assume that $m>1$. As $\Phi(G)=1$, we have $\mbf{F}(G)=M$ by Gasch\"{u}tz' theorem \cite[Theorem~III.4.5]{Hu67}, and hence $\Cen{G}{M}=M$ by \cite[Theorem~III.4.2 b)]{Hu67}.

First assume that $n=2$, where $n$ is the minimal size of a noncentral class of $G$. Let \mbox{$x\in G$} be such that $|x^G|=[G:\Cen{G}{x}]=2$. Then $\Cen{G}{x}\unlhd G$, and so $M \le \Cen{G}{x}$ as $M$ is the unique minimal normal subgroup of $G$. Hence, $x\in \Cen{G}{M}=M$. Since $\langle x^G\rangle$ is a normal subgroup of~$G$ and $x\in M$, we must have $\langle x^G\rangle=M$. As $|x^G|=2$, $M=\langle x,x^t\rangle$ for some $t\in G\setminus \Cen{G}{x}$. In particular, $|M|=p^2$ as $M$ is noncyclic elementary abelian and generated by two elements. Now, $M=\Cen{G}{M}=\Cen{G}{x}\cap \Cen{G}{x^t}=\Cen{G}{x}\cap \Cen{G}{x}^t=\Cen{G}{x}$ as $\Cen{G}{x}$ is normal in $G$. Thus, $[G:M]=[G:\Cen{G}{x}]=2$, and so $|G|=2p^2$. If $p=2$, then $G$ is a finite $2$-group, so it is nilpotent and hence it is supersolvable. Therefore, we may assume that $p\ge 3$. Since $|G|=2p^2$, $G$ has a cyclic Sylow $2$-subgroup $Q=\langle g\rangle$, where $g$ is an involution, acting on the unique minimal normal subgroup $M \cong C_p\times C_p$ of $G$ by conjugation. But this is impossible. To see this, let $1\neq a\in M$. We have $(aa^g)^g=a^ga=aa^g$, so $\langle aa^g\rangle$ is a normal subgroup of $G$. If $a a^g \neq 1$, then $\gen{a a^g}$ is a nontrivial cyclic normal subgroup of $G$ properly contained in the noncyclic group $M$, contradicting the fact that $M$ is minimal normal. On the other hand, if $aa^g=1,$ then $a^g=a^{-1}$ and $\langle a\rangle$ is again a nontrivial normal subgroup of $G$ properly contained in $M$.

Therefore, may assume that $n\ge 3$. Since $d(G)/\til{q}_h(G) \le (1/3)/(1/12) = 4$, we may choose to set $N=8$ in Lemma~\ref{l:c-inequality}. Since $3^{h-1}\leq n^{h-1} \le N=8$, we must have  $h=2$.
We have
\begin{align*}
   \frac{11}{72}=\til{q}_2(\mr{A}_4)< \til{q}_2(G) \leq \frac{1}{|G|}+\frac{d(G)}{n}\leq \frac{1}{|G|}+\frac{1}{9}.
\end{align*}
Simplifying this inequality, we obtain $|G|<24$. Since $G$ is non-supersolvable with trivial center, the only possibility for $G$ is $\mr{A}_4$. This is a contradiction.

There remains to consider case (2) of Lemma~\ref{l:c-inequality}. Since $\Zen{G}=1$ by the first paragraph and we chose $N=8$, we have
\begin{align*}
    |G| \leq \frac{8}{\frac{1}{12}\cdot 9 - \frac13} = \frac{96}{5}<20.
\end{align*}
The only non-supersolvable group of order less than $20$ with trivial center is again $\mr{A}_4$, so this contradiction completes the proof of the theorem.
\qed
\vspace{1em}

Now we prove the solvability criterion of Theorem~\ref{t:ch-criteria}(d).

\vspace{1em}
\noindent\emph{Proof of Theorem~\ref{t:ch-criteria}(d).}
Assume by contradiction that $G$ is a counterexample to the claim with minimal order. Since $\til{q}_1(G)=d(G)$, we must have $h\ge 2$. We have
\begin{align*}
    \frac{1}{60} < \til{q}_h(\mr{A}_5) < \til{q}_h(G) \le d(G) \le \frac{1}{12}.
\end{align*}
As $\til{q}_h(G/\mr{Sol}(G)) \ge \til{q}_h(G)$ by Lemma~\ref{lem:quotient}, the quotient $G/\mr{Sol}(G)$ is solvable if the solvable radical $\mr{Sol}(G)$ is nontrivial by the minimality of $|G|$. But then $G$ is solvable, a contradiction. Therefore, we may assume that $|\mr{Sol}(G)|=1$.

Let $n$ be the minimal nontrivial class size of $G$. If $n$ is a prime power, then the subgroup generated by that class is a solvable normal subgroup of $G$ by \cite{Kazarin}. As $G$ has no nontrivial solvable normal subgroups, we must have $n \ge 6$. Since $d(G)/\til{q}_h(G) \le (1/12)/(1/60) = 5$, we may take $N=5$ in Lemma~\ref{l:c-inequality}. This means that $n^{h-1} \le N$ never occurs, and we need only consider case (2) of Lemma~\ref{l:c-inequality}; that is,
\begin{align*}
    |G| \leq d(G)/\til{q}_h(G)^2 < 300.
\end{align*}
The only nonsolvable groups of order less than $300$ with trivial solvable radical are $\mr{A}_5$, $\mr{S}_5$, and $\mr{PSL}_3(2)$. The conjugacy class sizes of $\mr{A}_5$ are $1, 12, 12, 15, 20$; those of $\mr{S}_5$ are $1, 10, 15, 20, 20, 24, 30$; and those of $\mr{PSL}_3(2)$ are $1, 21, 24, 24, 42, 56$. Then one can check by direct calculation that $\til{q}_h(\mr{PSL}_3(2)) < \til{q}_h(\mr{S}_5) < \til{q}_h(\mr{A}_5)$.
This means $G$ is not isomorphic to $\mr{A}_5$, $\mr{S}_5$, or $\mr{PSL}_3(2)$, and we are done.
\qed

\vspace{1em}
We now prove Theorem~\ref{t:pclosed-dual}. For a prime $p$ and a positive integer $h$, define
\begin{align*}
    \til{\gamma}(h,p)=\frac{1}{p(p+1)} \Big(1+\frac{1}{p^{h-1}}+\frac{p-1}{(p+1)^{h-1}} \Big).
\end{align*}

\vspace{1em}

\noindent\emph{Proof of Theorem~\ref{t:pclosed-dual}}. Since $\til{\gamma}(h,2) = \til{q}_h(\textrm{S}_3)$, if $p=2$, then $G$ is nilpotent by Theorem~\ref{t:ch-criteria}(b) and thus has a normal Sylow $2$-subgroup. Therefore, we may assume $p\ge 3$. As $\til{q}_1(G)=d(G)$ and $\til{\gamma}(1,p)=1/p$, the theorem follows from Lemma~\ref{l:dG}(c) when $h=1$. So we may assume that $h \ge 2$. Let $G$ be a counterexample to the theorem with minimal order. Then 
\begin{align*}
    \frac{1}{p}\ge d(G)\ge \til{q}_h(G) > \til{\gamma}(h,p)>\frac{1}{p(p+1)}
\end{align*}
but $G$ is not $p$-closed. Let $P$ be a Sylow $p$-subgroup of $G$. We may assume that $O_p(G)=1$ as otherwise, by Lemma~\ref{lem:quotient} and the minimality of $|G|$, we deduce that $G/O_p(G)$ has a normal Sylow $p$-subgroup and so does $G$, a contradiction. Next, assume that $\Zen{G}>1$. Then $\Zen{G}$ is a $p'$-group, and $G/\Zen{G}$ has a normal Sylow $p$-subgroup $P\Zen{G}/\Zen{G}$ again by Lemma~\ref{lem:quotient}. It follows that $P\unlhd P\Zen{G}\unlhd G$ which implies that $P\unlhd G$, a contradiction. Therefore, $\Zen{G}=1$.

Next, we claim that $G$ has a unique minimal normal subgroup. Suppose that $G$ has two distinct minimal normal subgroups, say $N_1$ and $N_2$. By Lemma \ref{lem:quotient} and the minimality of $|G|$, both $G/N_1$ and $G/N_2$ are $p$-closed. Since $G$ embeds into $G/N_1\times G/N_2$, the group $G$ is $p$-closed as subgroups and direct products of $p$-closed groups are $p$-closed. It follows that $G$ has a unique minimal normal subgroup $N$, and $N\cong S^m$ for some simple group $S$ and some integer $m\ge 1$. 

Let $H=\Norm{G}{P}$. Since $PN\unlhd G$ by Lemma~\ref{lem:quotient} and the minimality of $|G|$, we have $G=HN$ by Frattini's argument.

\medskip
\textbf{Claim $1$:} \emph{If $C$ is a noncentral conjugacy class of $G$, then $|C|\ge p$.}

Let $C_0=x^G$ be a noncentral conjugacy class of $G$ that has minimal size. We claim that $|C_0|\ge p.$ Suppose by contradiction that $|C_0|=[G:\Cen{G}{x}]<p$. Let $U=\cap_{g\in G}\Cen{G}{x}^g$ be the core of $\Cen{G}{x}$ in $G$. Note that $U\unlhd G$ and $[G:U]=|C_0|\cdot [\Cen{G}{x}:U].$ As $G/U$ embeds into the symmetric group $\mr{S}_t$, where $t=|C_0|=[G:\Cen{G}{x}]<p$, $|G/U|$ is not divisible by $p$. Since $U \nrm G$ and Sylow subgroups are conjugate by Sylow's theorems, $P\leq U$. Hence $U$ is nontrivial and thus $N\leq U\leq \Cen{G}{x}.$ In particular, $x\in \Cen{G}{N}$.

If $N$ is nonabelian, then $\Cen{G}{N}=1$, which leads to a contradiction. Now assume that $N$ is abelian. Then $S\cong C_r$ for some prime $r$ and so $N$ is an elementary abelian $r$-group of order $r^m$. Since $G=HN$ and $N$ is the unique minimal normal subgroup of $G$, we must have $\Phi(G)=1$ and that $\mbf{F}(G)=O_r(G)$ is a product of minimal normal subgroups of $G$ by Gasch\"{u}tz' theorem \cite[Theorem~III.4.5]{Hu67}. Thus $\mbf{F}(G)=N$ and so $\Cen{G}{N}=N$ by \cite[Theorem~III.4.2 b)]{Hu67}. It follows that $x\in N$. Furthermore, since $P\leq \Cen{G}{x}$, we have $x\in \Cen{N}{P}.$ Observe that $P$ acts coprimely and nontrivially on $N$. By Fitting's Theorem \cite[Theorem~4.34]{I08}, $N=[N,P]\times \Cen{N}{P}.$ Note that $\Norm{N}{P}=N\cap \Norm{G}{P}=N\cap H$. As $N\unlhd G$, we have $H\cap N\unlhd H$ and since $N$ is abelian, $H\cap N\unlhd N$. It follows that $H\cap N\unlhd G=NH$, and by the minimality of $N$, we have $N\cap H=1$. Now $\Cen{N}{P}\leq \Norm{N}{P}=1$, which implies that $x\in \Cen{N}{P}=1$, a contradiction.

Therefore, we have shown that $|C_0|\ge p$, and thus since $\Zen G=1$, for each nontrivial conjugacy class $C$ of $G$, we have $|C|\ge p.$

\medskip
\textbf{Claim $2$:} \emph{$|G|<2p(p+1)$ and $|P|=p$. Let $n_p=[G:\Norm{G}{P}].$ Then $n_p=1+p$ or $1+2p.$ In particular,  $H=P$ and $G=PN$.}

Assume by contradiction that $|G|\ge 2p(p+1)$.  The hypothesis yields:
 \[\frac{1}{p(p+1)}<\til{q}_h(G)\leq \frac{1}{|G|}(1+\sum_{1\neq C\in \textrm{Cl}(G)}\frac{1}{|C|^{h-1}})\leq  \frac{1}{|G|}(1+\frac{k(G)-1}{p^{h-1}})= \frac{1}{|G|}(1-\frac{1}{p^{h-1}})+\frac{d(G)}{p^{h-1}}.\]
Then \[\frac{1}{p(p+1)}<\frac{1}{2p(p+1)}+\frac{d(G)}{p^{h-1}}.\]
Since $d(G)\leq 1/p$, we obtain $p^h< 2p(p+1)$, which implies $h=2$ as $p\ge 3$.

For $h=2$, the hypothesis yields (noting that $d(G)\leq 1/p$):
\begin{align*}
    \til{\gamma}(2,p) = \frac{1}{p(p+1)} \Big(1+\frac{1}{p}+\frac{p-1}{p+1} \Big)<\til{q}_h(G)\leq \frac{1}{|G|} \Big(1-\frac{1}{p}\Big)+\frac{d(G)}{p}\leq \frac{1}{|G|} \Big(1-\frac{1}{p} \Big)+\frac{1}{p^2}.
\end{align*}
Simplifying, we obtain $|G|<(p+1)^2<2p(p+1)$, a contradiction.

Assume that $|P|\ge p^2.$ 
If $N$ is nonabelian, then $|N|$ is divisible by $4$ and $|G|\ge 4p^2>2p(p+1)$, a contradiction. Assume that $N$ is abelian. Since $\Cen{G}{N}=N$, $P$ acts nontrivially and coprimely on~$N$. Hence $|N|>p$ as $P$ has at least one nontrivial orbit on $N\setminus \{1\}.$ Then $H\cap N=1$ by the argument in Claim 1, and thus $|G|=|H||N|\ge p^2|N|>p^3>2p(p+1)$. In both cases, we obtain a contradiction. Thus, $|P|=p$ as claimed.

By Sylow's theorem, we know that $n_p\equiv 1$ mod $p$, so that $n_p=1+kp$ for some positive integer~$k$. Since $n_p=[G:H]\leq [G:P]<2(p+1)$, we have $n_p<2(p+1)$ and so $n_p\in \{1+p,1+2p\}.$

Now since $|G|=|H|n_p\in \{(1+p)|H|,(1+2p)|H|\}$ and $|G|<2p(p+1),$ we must have $|H|<2p$, forcing $|H|=p$ and thus $H=\Cen{G}{P}=P$. In particular, $G=PN$.

\medskip
\textbf{Claim $3$:} \emph{$N$ is abelian.}

Suppose by contradiction that $N$ is nonabelian. Then $N\cong S^m$ for some nonabelian simple group $S$. By the uniqueness of $N$, we deduce that $G$ has a trivial solvable radical.

If $|C|=p$ for some class $C$, then by \cite{Kazarin}, $\langle C\rangle$ is a normal solvable subgroup of $G$, which is impossible. Thus $|C|\ge p+1$ for every nontrivial class $C$ of $G$. 

If $p$ divides $|S|$, then $p=|G|_p \ge (|S|_p)^m$ means that $m=1$ and $G=PS=S$ is simple. However, in this case, $|G|\ge 2p(p+1)$ by the proof of Lemma \ref{lem1}. Thus $p$ does not divide  $|S|$ and so $N$ is a $p'$-group. So $|G|=p|S|^m$ with $m=1$ or $m=p$. (If $N$ is not simple, then $m=p$ since $P$ acts transitively on the simple factors of $N$ as $N$ is a minimal normal subgroup of $G$.)

Now  $|N|=[G:H]=n_p\in \{1+p,1+2p\}$, so we have $|N|=|S|^m=1+p$ or $1+2p$. Since $P$ has order $p$ and $\Cen{G}{P}=P$ by the last step, the only point of $N$ fixed by $P$ under conjugation is the identity and every other orbit has size $p$. This implies that $P$ has at most 3 orbits in its conjugation action on $N$, which is impossible since $|N|$ is divisible by at least $3$ distinct primes by Burnside's $p^aq^b$-theorem,  so $N$ must have at least $4$ orbits under the action of $P$.

\medskip
\textbf{Claim $4$:} \emph{The final contradiction.}

We now have that $G=PN$ with $|P|=p$, and $N$ is an elementary abelian $r$-group of order $r^m$ for some prime $r$ and $r^m=|N|=1+p$ or $1+2p$.  Since $\Zen{G}=1$, we see that $G$ is a Frobenius group with Frobenius kernel $N$ and complement $P$. Therefore, $G$ has the following class structure: the $r^m-1$ elements of order $r$ split into classes of size $p$, and the $r^m(p-1)$ elements of order $p$ split into classes of size $r^m$. Since $r^m=1+kp$ for $k=1$ or $k=2$, we have
\begin{align*}
\til{q}_h(G)=\frac{1}{|G|}\Big(1+\frac{(r^m-1)/p}{p^{h-1}}+\frac{(p-1)}{(r^{m})^{h-1}}\Big) = \frac{1}{p(1+kp)}\Big(1+\frac{k}{p^{h-1}}+\frac{(p-1)}{(1+kp)^{h-1}}\Big).
\end{align*}
If $|N|=1+p$, then substituting $k=1$ into the equation above yields $\til{q}_h(G)=\til{\gamma}(h,p)$, a contradiction. So $r^m=|N|=1+2p$, and we can calculate $\til{q}_h(G)$ by substituting $k=2$ into the equation above.

Since \[\frac{1}{p(p+1)}<\til{q}_h(G)\leq \frac{1}{|G|}+\frac{d(G)}{p^{h-1}}\leq \frac{1}{p(1+2p)}+\frac{1}{p^{h}},\] we deduce that $p^h<(1+p)(1+2p)$, forcing $h=3$ and $p=3$, or $h=2$. In the former case, we calculate directly that $\til{\gamma}(3,3)>\til{q}_3(G)$, a contradiction. Similarly, if $h=2$, then we calculate directly that we must have
\begin{align*}
    \frac{1}{p(p+1)} \Big(1+\frac{1}{p}+\frac{p-1}{p+1} \Big) = \til{\gamma}(2,p)<\til{q}_2(G) = \frac{1}{p(1+2p)} \Big(1+\frac{2}{p}+\frac{p-1}{1+2p} \Big).
\end{align*}
This can be simplified to $5p^4+2p^3-3p^2-3p<1$, which is impossible since $p\ge 3$. The proof is now complete.
\qed

\section{Introduction to Dijkgraaf--Witten theory.}\label{s:tqft}

For mathematicians, a topological quantum field theory (TQFT) is a machine for generating topological invariants of manifolds. For example, Witten \cite{W89} showed that the Jones polynomial is a topological invariant of manifolds in a certain TQFT. Atiyah \cite{At89} axiomatized topological quantum field theories in the language of category theory as a functor from a particular category of bordisms to the category of vector spaces. The invariants that we consider arise in the context of Dijkgraaf--Witten theory \cite{Di95,DW90} in one space dimension and one time dimension. Such a TQFT determines a commutative Frobenius algebra that gives rise to topological invariants of surfaces. The new contribution of this paper is that, upon assuming this Frobenius algebra is the center of a complex group algebra, we instead consider these invariants as invariants of groups. It should be noted, however, that Dijkgraaf--Witten theory has already been used to study finite groups; for example, row and column sums in their character tables \cite{K22,PRR24,RS22}. In this section, we derive these topological invariants in an elementary, explicit way, without assuming any previous knowledge of quantum field theory. For a similarly gentle introduction, we recommend Dijkgraaf's 2015 lecture at the Institute for Advanced Study~\cite{Di15}. 

\begin{table}[t]
    \centering
    \begin{tabular}{ >{\centering\arraybackslash}m{0.3\textwidth} | >{\centering\arraybackslash}m{0.3\textwidth} | >{\centering\arraybackslash}m{0.3\textwidth} } \hline
    Bordism & Map & $V=\Zen{\C G}$ \\ \hline
    %
    %
    \begin{tikzpicture}[tqft/cobordism/.style={draw}, tqft/every boundary component/.style={draw},rotate=90,transform shape]
    \pic [tqft/reverse pair of pants, name];
    \end{tikzpicture} 
    & $V \otimes V \rightarrow V$ 
    & $e_\chi \cdot e_{\chi'} = \delta_{\chi,\chi'} e_\chi$ \\ \hline
    %
    %
    \begin{tikzpicture}[tqft/cobordism/.style={draw}, tqft/every boundary component/.style={draw},rotate=90,transform shape]
    \pic[tqft/reverse pair of pants, name=rpants2];
    \pic [tqft/cup, name=cup, at=(rpants2-outgoing boundary 1), anchor = incoming boundary 1];
    \end{tikzpicture}
    & $\gen{\;}:V \otimes V \rightarrow \C$
    & $\gen{e_\chi,e_{\chi'}} = \delta_{\chi,\chi'} \big(\frac{\chi(1)}{|G|}\big)^2$ \\ \hline
    %
    %
    \begin{tikzpicture}[tqft/cobordism/.style={draw}, tqft/every boundary component/.style={draw},rotate=90,transform shape]
    \pic [tqft/cap, name];
    \end{tikzpicture} 
    & \makecell{$\C \rightarrow V$ \\ $1 \mapsto e$ }
    & $e = \sum\limits_{\chi \in \Irr{G}} e_\chi$\\ \hline
    %
    %
    \begin{tikzpicture}[tqft/cobordism/.style={draw}, tqft/every boundary component/.style={draw},rotate=90,transform shape]
    \pic [tqft/cup, name];
    \end{tikzpicture} 
    & $V \rightarrow \C$
    & $e_\chi \mapsto \big(\frac{|G|}{\chi(1)}\big)^2$\\ \hline
    %
    %
    \begin{tikzpicture}[tqft/cobordism/.style={draw}, tqft/every boundary component/.style={draw},rotate=90,transform shape]
    \pic [tqft/pair of pants, name=rpants];
    \end{tikzpicture}
    & $V \rightarrow V \otimes V$
    & $e_\chi \mapsto \big(\frac{|G|}{\chi(1)}\big)^2 e_\chi \otimes e_\chi$ \\ \hline
    \end{tabular}
    \caption{Bordisms giving rise to linear maps of vector spaces.}
    \label{tab:1}
\end{table}

Let $X$, $Y$ be two closed ($n-1$)-dimensional manifolds. A bordism from $X$ to $Y$ is a quadruple $(Z,p, f_X,f_Y)$ where $Z$ is an $n$-manifold with boundary $\delta Z$, $p:\delta Z\longrightarrow \{0,1\}$ is a continuous map (partition of the boundary into two components), $f_X:X\longrightarrow p^{-1}(0)$, $f_Y:Y\longrightarrow p^{-1}(1)$ are diffeomorphisms. (The empty set is considered an $n-1$-manifold for all $n$.) Two such bordisms $Z$, $Z'$ from $X$ to $Y$ are considered equivalent if there is a diffeomorphism from $Z$ to $Z'$ respecting the maps $p$ and $f$ in the obvious way. Declaring Mor (X,Y) to be the set of all equivalence classes of bordisms from $X$ to $Y$, we get the structure of a category $B(n-1)$ on closed $(n-1)$-manifolds. The composition of two bordisms is obtained by gluing them. This category $B(m)$ is a symmetric monoidal category with disjoint union as the monoidal product.  Any covariant functor of symmetric monoidal categories from $B(m)$ to vector spaces $\mr{Vect}_k$ over a field $k$, with tensor product as the monoidal product, is an example of a topological quantum field theory.
 
When $n=2$, closed $1$-manifolds are just disjoint unions of a finite number of circles $S^1$. So a TQFT $T$ of the above type on objects of $B(1)$ is determined by specifying $T(S^1)=V$. The images of morphisms via $T$ are linear maps between appropriate tensor powers of~$V$. For example, a cylinder is a bordism between two circles, and it corresponds to the identity map $\mr{id}:V \rightarrow V$. 
Perhaps most importantly, the ``pair of pants'' corresponds to a multiplication of quantum states $V \otimes V \rightarrow V$, which gives $V$ the structure of an algebra. This multiplication is commutative, which we can see from topological considerations: Interchanging the ``legs'' in a pair of pants (which corresponds, on the vector space side, to switching the order of the tensor product) is a diffeomorphism, so two such bordisms are equivalent. The bordisms giving rise to a multiplicative identity $e \in V$, a comultiplication $V \rightarrow V \otimes V$ and a bilinear form $\gen{\;}:V \otimes V \rightarrow \C$ are pictured in Table~\ref{tab:1}. By considering appropriate bordisms, it is not difficult to show that the multiplication is associative and the bilinear form satisfies the identity $\gen{a,b\cdot c} = \gen{a \cdot b, c}$ of a Frobenius algebra. A genus-$h$ surface can be decomposed into a composition of the bordisms in Table~\ref{tab:1}, giving rise to a map $\C \rightarrow \C$. Since the map is linear, it is defined by the image of~$1$, and this complex number is a topological invariant of the surface.


\begin{figure}[t]
    \centering
    \begin{minipage}{0.33\textwidth}
        \centering
        \begin{tikzpicture}[tqft/cobordism/.style={draw}, tqft/every boundary component/.style={draw},rotate=90,transform shape]
        \pic[tqft/cap, name=cap];
        \pic [tqft/reverse pair of pants, name, at=(cap-outgoing boundary 1), anchor = incoming boundary 2];
        \end{tikzpicture}
        \caption*{(a)}
    \end{minipage}%
    \hfill 
    \begin{minipage}{0.33\textwidth} 
    \centering
        \begin{tikzpicture}[tqft/cobordism/.style={draw}, tqft/every boundary component/.style={draw},rotate=90,transform shape]
        \pic[tqft/cap, name=cap];
        \pic [tqft/reverse pair of pants, name=rpants, at=(cap-outgoing boundary 1), anchor = incoming boundary 2];
        \pic [tqft/cup, name=cup, at=(rpants-outgoing boundary 1), anchor = incoming boundary 1];
        \end{tikzpicture}
        \caption*{(b)}
    \end{minipage}
    \hfill
    \begin{minipage}{0.33\textwidth}
        \centering
        \begin{tikzpicture}[tqft/cobordism/.style={draw}, tqft/every boundary component/.style={draw},rotate=90,transform shape]
        \pic[tqft/pair of pants, name=rpants2];
        \pic [tqft/cup, name=cup, at=(rpants2-outgoing boundary 2), anchor = incoming boundary 1];
        \end{tikzpicture}
        \caption*{(c)}
    \end{minipage}
    \par
    \vspace{1em}
    \begin{minipage}[c]{1\textwidth}
    \begin{center}
    \begin{tabular}{ >{\centering\arraybackslash}m{0.54\textwidth}  >{\centering\arraybackslash}m{0.01\textwidth}  >{\centering\arraybackslash}m{0.27\textwidth} }
    \begin{tikzpicture}[tqft/cobordism/.style={draw}, tqft/every boundary component/.style={draw},rotate=90,transform shape]
    \pic [tqft/cap, name=s1];
    \pic [tqft/pair of pants, name=s2,at=(s1-outgoing boundary 1), anchor = incoming boundary 1];
    \pic [tqft/reverse pair of pants, name=s3, at=(s2-outgoing boundary 1), anchor = incoming boundary 1];
    \pic [tqft/pair of pants, name=s4,at=(s3-outgoing boundary 1), anchor = incoming boundary 1];
    \pic [tqft/reverse pair of pants, name=s5, at=(s4-outgoing boundary 1), anchor = incoming boundary 1];
    \end{tikzpicture} & 
    $\cdots$ &
    \begin{tikzpicture}[tqft/cobordism/.style={draw}, tqft/every boundary component/.style={draw},rotate=90,transform shape]
    \pic [tqft/pair of pants, name=s2,at=(s1-outgoing boundary 1), anchor = incoming boundary 1];
    \pic [tqft/reverse pair of pants, name=s3, at=(s2-outgoing boundary 1), anchor = incoming boundary 1];
    \pic [tqft/cup, name=s4,at=(s3-outgoing boundary 1), anchor = incoming boundary 1];
    \end{tikzpicture}
    \end{tabular}
    \end{center}
    \caption*{(d)}
    \end{minipage}
    \captionof{figure}{Bordisms for calculating (a) the ``cup'', (b) the ``cap'' and (c) the comultiplication. Decomposing the genus $h$-surface in (d) into bordisms calculated in Table~\ref{tab:1} yields an invariant $\mc{Q}_h(G)$ of the surface.} 
    \label{fig:1}
\end{figure}

Now we assume that $V=\Zen{\C G}$ is the center of the complex group algebra of a finite group~$G$. We compute the maps in terms of a basis for $\Zen{\C G}$ that is indexed by the irreducible complex characters $\chi \in \Irr{G}$ of $G$; namely, the primitive central idempotents 
\begin{align*}
    e_\chi = \frac{\chi(1)}{|G|} \sum_{g \in G} \q{\chi(g)} g.
\end{align*}
In this basis, the multiplicative identity decomposes as $e = \sum_{\chi \in \Irr{G}} e_\chi$, and the multiplication takes the form $e_\chi \cdot e_{\chi'} = \delta_{\chi,\chi'} e_\chi$. We define the bilinear form to be
\begin{align*}
    \gen{e_\chi,e_{\chi'}} = \delta_{\chi,\chi'} \bigg(\frac{\chi(1)}{|G|}\bigg)^2,
\end{align*}
which is the coefficient of the identity in the product $e_\chi \cdot e_{\chi'}$ times~$1/|G|$. This is the so-called special Frobenius algebra structure on $\Zen{\C G}$. 

To show that the ``cup'' in the third row of Table~\ref{tab:1} gives rise to a multiplicative identity of~$V$, note that the bordism in Fig.~\ref{fig:1}(a) is homeomorphic to the cylinder, which is the identity map. To show that the ``cap'' in the fourth row of Table~\ref{tab:1} sends $e_\chi$ to $(\chi(1)/|G|)^2$, note that the cap is homeomorphic to adding a cup to the bordism for the bilinear form, as in Fig.~\ref{fig:1}(b). Finally, to show that the comultiplication is $e_\chi \mapsto (|G|/\chi(1))^2 e_\chi \otimes e_\chi$ in the last row of Table~\ref{tab:1}, note that capping off the comultiplication is homeomorphic to the cylinder, which represents the identity map, as in Fig.~\ref{fig:1}(c).

Now, we compute the map $\C \rightarrow \C$ for a genus-$h$ surface by composing the maps in Table~\ref{tab:1}:
\begin{align*}
    1 & 
    \overset{\scalebox{0.45}{\begin{tikzpicture}[tqft/cobordism/.style={draw}, tqft/every boundary component/.style={draw},rotate=90,transform shape]
    \pic [tqft/cap, name];
    \end{tikzpicture}}}{\longmapsto}
    \sum_{\chi \in \Irr{G}} e_\chi
    \overset{\scalebox{0.2}{\begin{tikzpicture}[tqft/cobordism/.style={draw}, tqft/every boundary component/.style={draw},rotate=90,transform shape]
    \pic [tqft/pair of pants, name];
    \end{tikzpicture}}}{\longmapsto}
    \sum_{\chi \in \Irr{G}} \bigg( \frac{|G|}{\chi(1)} \bigg)^2 e_\chi \otimes e_\chi
    \overset{\scalebox{0.2}{\begin{tikzpicture}[tqft/cobordism/.style={draw}, tqft/every boundary component/.style={draw},rotate=90,transform shape]
    \pic [tqft/reverse pair of pants, name];
    \end{tikzpicture}}}{\longmapsto}
    \sum_{\chi \in \Irr{G}} \bigg( \frac{|G|}{\chi(1)} \bigg)^2 e_\chi \\
    & \overset{\scalebox{0.2}{\begin{tikzpicture}[tqft/cobordism/.style={draw}, tqft/every boundary component/.style={draw},rotate=90,transform shape]
    \pic [tqft/pair of pants, name];
    \end{tikzpicture}}}{\longmapsto}
    \cdots
    \overset{\scalebox{0.2}{\begin{tikzpicture}[tqft/cobordism/.style={draw}, tqft/every boundary component/.style={draw},rotate=90,transform shape]
    \pic [tqft/reverse pair of pants, name];
    \end{tikzpicture}}}{\longmapsto}
    \sum_{\chi \in \Irr{G}} \bigg( \frac{|G|}{\chi(1)} \bigg)^{2h} e_\chi
    \overset{\scalebox{0.4}{\begin{tikzpicture}[tqft/cobordism/.style={draw}, tqft/every boundary component/.style={draw},rotate=90,transform shape]
    \pic [tqft/cup, name];
    \end{tikzpicture}}}{\longmapsto}
    \sum_{\chi \in \Irr{G}} \bigg( \frac{|G|}{\chi(1)} \bigg)^{2h-2}=\mc{Q}_h(G).
\end{align*}

The number $\mc{Q}_h(G)$ is the topological invariant of a genus-$h$ surface in this TQFT. In this paper, we considered a more well-behaved invariant of groups by scaling $\mc{Q}_h(G)$ by an appropriate factor of~$|G|$, namely $q_h(G)=1/|G| \sum_{\chi \in \mr{Irr}(G)} (1/\chi(1))^{2h-2}$.

\section*{Acknowledgment} The authors would like to thank Dr.\@ Alessandro Mariani and Prof.\@ Marcin Mazur for comments on a previous version of the paper, as well as the anonymous referee for their close reading of the paper and helpful comments.



\footnotesize

\textsc{Department of Mathematics and Statistics, Binghamton University, Binghamton, NY 13902-6000, USA}\par\nopagebreak \textit{E-mail address}: \texttt{cschroe2@binghamton.edu}

\vspace{\baselineskip}

\textsc{Department of Mathematics and Statistics, Binghamton University, Binghamton, NY 13902-6000, USA}\par\nopagebreak \textit{E-mail address}: \texttt{htongvie@binghamton.edu}

\end{document}